\documentclass{article}
\usepackage{makeidx} 
\usepackage[utf8]{inputenc}
\usepackage{hyperref}
\usepackage{colortbl}
\usepackage[margin=1.25in]{geometry}
\usepackage[normalem]{ulem}
\usepackage{lscape}
\usepackage{wrapfig}
\usepackage{booktabs}
\usepackage{bbm}
\usepackage{xcolor}
\usepackage{graphicx}
\usepackage{amssymb}
\usepackage{amsbsy}
\usepackage{amsmath}
\usepackage{amsthm}
\usepackage{amsfonts}
\usepackage{enumitem}
\usepackage{subcaption}
\usepackage{caption}
\usepackage{listings}
\usepackage{authblk}
\usepackage[]{natbib}
\usepackage{enumitem}
\usepackage{algorithm}
\usepackage{tabularx}
\usepackage{algpseudocode}
\usepackage{accents}
\usepackage{multicol}
\usepackage{pgfplots}
\pgfplotsset{compat=1.18}
\usetikzlibrary{patterns, arrows.meta}
\usepgfplotslibrary{fillbetween}

\newcommand{\ignore}[1]{}

\hypersetup{
    colorlinks=true,
    linkcolor=blue,
    filecolor=magenta,      
    urlcolor=cyan,
}
\newtheorem{example}{Example}

\newtheorem{proposition}{Proposition}
\newtheorem{definition}{Definition}

\DeclareMathOperator{\relint}{relint}

\DeclareMathOperator{\dom}{dom}
\DeclareMathOperator{\epi}{epi}
\DeclareMathOperator{\conv}{conv}

\DeclareMathOperator{\Proj}{Proj}

\newcommand{\br}[1]{\{ {#1} \}}
\newcommand{\mc}[1]{\mathcal{#1}}

\newcommand{\ti}[1]{\tilde{#1}}
\newcommand{\mb}[1]{\mathbb{#1}}

\newcommand{\cbr}[1]{\left( #1 \right)}

\newcommand{\tcb}[1]{{#1}}

\newcommand{\ul}[1]{\underline{#1}}
\newcommand{\autotag}{\tag{\theequation}\stepcounter{equation}}
\newcommand{\ol}[1]{\overline{#1}}
\newcommand{\co}{\operatorname{co}}

\newcommand{\un}[1]{\underline{#1}}

\newcommand{\pmat}[1]{\begin{pmatrix}#1\end{pmatrix}}
\newcommand{\mbf}[1]{\mathbf{#1}}

\title{Normalization of ReLU Dual for Cut Generation in Stochastic Mixed-Integer Programs}
\date{\today}
\author{Akul Bansal}
\author{Simge K\"u\c{c}\"ukyavuz}
\affil{Industrial Engineering and Management Sciences \\ Northwestern University \\
{\texttt{akul@u.northwestern.edu, simge@northwestern.edu}}}

\begin{document}
\maketitle

\begin{abstract}
We study the Rectified Linear Unit (ReLU) dual, an existing dual formulation for stochastic programs that reformulates non-anticipativity constraints using ReLU functions to generate tight, non-convex, and mixed-integer representable cuts. While this dual reformulation guarantees convergence with mixed-integer state variables, it admits multiple optimal solutions that can yield weak cuts. To address this issue, we propose normalizing the dual in the extended space to identify solutions that yield stronger cuts. We prove that the resulting normalized cuts are tight and Pareto-optimal in the original state space. We further compare normalization with existing regularization-based approaches for handling dual degeneracy and explain why normalization offers key advantages. In particular, we show that normalization can recover any cut obtained via regularization, whereas the converse does not hold. Computational experiments demonstrate that the proposed approach outperforms existing methods by consistently yielding stronger cuts and reducing solution times on harder instances.
\end{abstract}

\section{Introduction}

We consider a multistage stochastic integer program (MSIP) defined on a scenario tree
$\mc{T}$ over stages $t = 1, \ldots, T$. Each node $n \in \mathcal{T}$
corresponds to a realization of uncertainty at a particular stage, and the set of
nodes at stage $t$ is denoted by $N_t$. For any node $n$, let $a(n)$ denote its
unique ancestor, $C(n)$ its set of children, and $q_{nm}$ the transition
probability from node $n$ to a child node $m \in C(n)$. At each node $n$, a $d_n$-dimensional state variable vector $x_n \in X_n \subseteq \mathbb{R}^{d_n}$ connects successive stages, and a local \tcb{variable vector} $y_n \in Y_n$ is specific to the subproblem at that node. The sets
$X_n$ and $Y_n$ may include integrality restrictions. Given the state $x_{a(n)}$ inherited
from its ancestor $a(n)$, the feasible decisions at node $n$ are given by the
polyhedron $H_n(x_{a(n)}) = \Big\{ (x_n, y_n) \;\big|\; T_n x_{a(n)} + W_n x_n + \overline{W}_n y_n = b_n \Big\},$ where $T_n, W_n, \overline{W}_n$ are technology and recourse matrices, and $b_n$ is a right-hand side vector of appropriate dimension. The MSIP can then be formulated as:
\begin{align} \label{prob:root}
\min_{x_1, y_1}\ 
& \Bigg\{ f_1(x_1, y_1)
+ \sum_{m \in C(1)} q_{1m}\, Q_m(x_1)
\ : \ 
(x_1, y_1) \in H_1(x_0) \cap (X_1 \times Y_1) \Bigg\},
\end{align}
where for any  node $n$, the associated value function is defined
recursively as
\begin{subequations} \label{prob:sub}
\begin{align}
Q_n(x_{a(n)}) 
= \min_{x_n, y_n}\ 
& f_n(x_n, y_n)
+ \sum_{m \in C(n)} q_{nm}\, Q_m(x_n), \\
& (x_n, y_n) \in H_n(x_{a(n)}) \cap (X_n \times Y_n). 
\end{align}
\end{subequations}
The objective function $f_n(\cdot)$ is linear, and the initial state $x_0$ is given \emph{apriori}.
For final-stage nodes $n \in N_T$, the set of children nodes $C(n)$ is empty, and the expected future cost
$\sum_{m \in C(n)} q_{nm}\, Q_m(x_n)$ is defined to be $0$. The function
$Q_n : \mathbb{R}^{d_{a(n)}} \to \mathbb{R} \cup \{+\infty\}$ takes value $\infty$
whenever the feasible set $H_n(x_{a(n)})$ is empty; we denote by
$\operatorname{dom}(Q_n)$ the set of incumbent states $x_{a(n)}$ for which the subproblem $Q_n$ is feasible. 

Although MSIPs admit a deterministic equivalent mixed-integer programming (MIP) formulation, the number of variables and constraints grows exponentially with the size of the scenario tree. This quickly overwhelms off-the-shelf solvers and motivates the use of decomposition methods for large-scale instances (see \citealt{romeijndersastochastic,kuccukyavuz2017introduction} for comprehensive reviews). A standard approach to decomposing the value function $Q_n(\cdot)$, is to reformulate the subproblem \eqref{prob:sub} as
\begin{subequations} \label{prob:sub_reform2}
    \begin{align*}
       Q_n \left(x_{a(n)} \right) = \min_{\substack{\\ x_n, y_n, z_n,\\ (\theta_m)_{m \in C(n)}}}\;& f_n(x_n, y_n) 
       + \sum_{m \in C(n)} q_{nm}\, \theta_{m}, \autotag  \\
       & (x_n, y_n) \in H_n(z_n) \cap (X_n \times Y_n), \autotag \label{sub_reform2:poly}\\
       & (x_n, \theta_m) \in \epi(Q_m),  &\forall m \in C(n), \autotag \label{sub_reform2:epi}\\
       & z_n = x_{a(n)}, \autotag \label{sub_reform2:copy}\\
       &z_n \in Z_{a(n)} \supseteq X_{a(n)},  \autotag \label{sub_reform2:copy_space}
    \end{align*}
\end{subequations}
where the set $\epi(Q_m)$ is defined as follows:
\begin{align} \label{eq:epi_Q}
    \epi(Q_m) = \br{(x_{a(m)}, \theta_m) \in \mb{R}^{d_{a(m)}} \times \mb{R}: \theta_m \geq Q_m(x_{a(m)}), x_{a(m)} \in \dom(Q_m)}.
\end{align}
When needed, we write $\epi_S(p)$ to denote the epigraph of function $p$ restricted to the set $S$. For instance, $\epi_{\dom(Q_m)}(Q_m)$ (or equivalently $\epi(Q_m)$ on $\dom(Q_m)$) denotes the epigraph in \eqref{eq:epi_Q}.

Because the value function $Q_m(\cdot)$ is not available in closed form, we approximate its epigraph $\epi(Q_m)$ with a polyhedral set $\Psi_m$ which is iteratively refined by adding cuts of the form $h_m^i(x_n, \theta_m) \ge 0$. At a given iteration $i$, the inequality $h_m^i(\cdot,\cdot) \ge 0$ serves either as an optimality cut, yielding a valid lower bound on $Q_m(\cdot)$, or as a feasibility cut that approximates the domain $\operatorname{dom}(Q_m)$. In the nested Benders’ method (\citealt{birge1985decomposition}), these cuts are computed by traversing the entire scenario tree, whereas sampling-based variants construct them using subsets of scenario paths at each iteration (\citealt{pereira1991multi,zou2019stochastic,fullner2025stochastic}).

To facilitate cut generation, it is common to introduce local variables $z_n$ together with copy constraints $z_n = x_{a(n)}$, and relax the domain with $z_n \in Z_{a(n)} \supseteq X_{a(n)}$. For instance, Benders cuts (\citealt{benders1962partitioning}) can be derived from the optimal value of the linear programming (LP) relaxation of subproblem \eqref{prob:sub_reform2} and the corresponding LP dual solution associated with the copy constraints \eqref{sub_reform2:copy}. When the value function is convex polyhedral, it admits an exact representation using finitely many Benders cuts (see \citealt{rahmaniani2017benders} for a review on classical Benders decomposition). In contrast, when decision variables are integer or binary, the value function $Q_{n}(\cdot)$ is typically nonlinear and nonconvex (\citealt{blair1982value}). For problems with purely integer or binary variables, several specialized decomposition techniques have been proposed, including disjunctive programming (\citealt{sen2005c, sen2006decomposition,qi2017ancestral}), Fenchel cuts (\citealt{ntaimo2013fenchel}), and parametric Gomory cuts (\citealt{gade2014decomposition,zhang2014finitely}). Additional structure has also been exploited in other settings, such as when the value function is monotone (\citealt{philpott2020midas}) or Lipschitz continuous (\citealt{ahmed2022stochastic,fullner2024lipschitz}).

More recently, there has been growing interest in decomposition methods with convergence guarantees for general MSIPs with {\em mixed-integer} state variables. For example, \citealt{van2024converging} study scaled cuts and show that the associated scaled-cut closures converge uniformly to the convex envelope of the expected recourse function. However, these cuts cannot be computed using scenario decomposition methods, and the computational performance of the multistage extension (\citealt{romeijnders2024benders}) has not been tested. 

In contrast, \citealt{zou2019stochastic} propose an alternative approach based on Lagrangian cuts derived by solving the Lagrangian dual obtained from relaxing the copy constraints \eqref{sub_reform2:copy}. Unlike Benders cuts, which only recover the LP relaxation of the value function, Lagrangian cuts can recover the closed convex envelope $\overline{{\co}}(Q_n)$ of $Q_n$ (see \citealt{chen2022generating} and \citealt{fullner2024lipschitz} for more details). Moreover, \citealt{zou2019stochastic} show that when the state variables are binary, Lagrangian cuts are sufficient to ensure tightness for $Q_n$, a property that, in turn, yields finite convergence of the decomposition method. To ensure convergence with non-binary state variables, \citealt{zou2019stochastic} propose encoding them with auxiliary binary variables. For continuous state variables, however, this discretization can dramatically expand the state space, increasing the cost of solving the Lagrangian dual and potentially making the problem computationally prohibitive (see \citealt{zou2019stochastic,fullnernew,yangyang2025} for implementation details and discussion).

\citealt{deng2024relu} strike a middle-ground approach by ensuring convergence for general MSIPs with mixed-integer state variables without inflating the state dimension. They propose the following reformulation of subproblem \eqref{prob:sub_reform2}
\begin{subequations} \label{sub_reform:relu}
\begin{align*}
    Q_n \left(x_{a(n)} \right) = \min_{\substack{\\ x_n, y_n, z_n,\\ (\theta_m)_{m \in C(n)}}}\;& f_n(x_n, y_n) 
       + \sum_{m \in C(n)} q_{nm}\, \theta_{m}, \\
    & \text{ s.t. }\eqref{sub_reform2:poly}, \eqref{sub_reform2:epi}, \eqref{sub_reform2:copy_space}, \autotag \label{sub_reform:orig_constr} \\
    & (z_{nk} - x_{a(n)k})^+ = 0, \quad (z_{nk} - x_{a(n)k})^- = 0, &\forall k \in [d_{a(n)}].\autotag \label{sub_reform:copy} 
\end{align*}
\end{subequations}
The notation $[d_{a(n)}]$ denotes the set $\br{1, \ldots d_{a(n)}}$. The copy constraints in \eqref{sub_reform2:copy} are written using ReLU functions in \eqref{sub_reform:copy}, where \((x)^+ := \max\{x,0\}\) and \((x)^- := \max\{-x,0\}\). These expressions are linearized using auxiliary variables that remain local to each subproblem, thereby preserving the state dimension. Relaxing the resulting copy constraints \eqref{sub_reform:copy} yields the ReLU dual, where the number of dual variables is bounded by twice the state dimension. Notably, the authors establish strong duality for this formulation, and the resulting ReLU-dual solutions can therefore be used to generate tight cuts even with mixed-integer state variables, which ensures asymptotic convergence for general MSIPs.

A related approach by \citealt{yangyang2025} also establishes asymptotic convergence for general MSIPs. Instead of using ReLU-based copy constraints in \eqref{sub_reform:copy}, they use the original copy constraints $z_n = x_{a(n)}$ in a lifted space. The ReLU-based copy constraints in \eqref{sub_reform:copy} can be interpreted as inducing a partition of the state space: along each dimension, the space is divided into regions to the left and right of the incumbent point, and the overall partition is obtained via the Cartesian product across all dimensions. The lifting procedure of \citealt{yangyang2025} generalizes this partition created by iteratively introducing binary state variables and expanding the state dimension.

All dual-based cut-generation schemes—whether based on the LP dual (as in classical Benders decomposition), the Lagrangian dual (\citealt{zou2019stochastic}), the ReLU dual of \citealt{deng2024relu}, or the lifted Lagrangian dual of \citealt{yangyang2025}—share a common challenge. They admit multiple optimal solutions, some of which can generate weak cuts. For instance, \citealt{bansal2024computational} show that coefficients of integer L-shaped cuts--known for their weak global approximation and resulting slow convergence---are one of the optimal solutions to the Lagrangian dual.

In the LP-dual/Benders setting, this issue has been studied extensively, and a range of alternative cut-generation rules has been proposed to mitigate weak cuts and improve separation. \citealt{magnanti1981accelerating} introduce a two-step procedure to obtain Pareto-optimal Benders cuts. \citealt{fischetti2010note} introduce a unified cut-generation framework, which yields both feasibility and optimality cuts through a single separation problem, and admits methodological flexibility in identifying unbounded rays for cut generation through the choice of normalization. Different normalization choices can yield cuts that are deep (\citealt{hosseini2021deepest}) and facet-defining or Pareto-optimal (\citealt{brandenberg2021refined}). \citealt{fullnernew} extend the work of \citealt{fischetti2010note} to the Lagrangian dual of the mixed-integer subproblems to obtain Lagrangian cuts with desirable properties.

For more general duals—such as the ReLU dual and the lifted Lagrangian dual—\citealt{deng2024relu,yangyang2025} propose mitigation strategies that extend the ideas of \citealt{magnanti1981accelerating} to their formulations. These strategies can be viewed as dual regularization: they seek to characterize the set of all optimal dual solutions and then select those that yield stronger cuts by optimizing a regularized objective over this set. The modified objective includes additional terms involving the dual variables, weighted by appropriately chosen objective coefficients. In particular, \citealt{deng2024relu} approximate the set of optimal dual solutions using a linear program. They derive objective coefficients so that the regularized dual remains bounded. The cut-generating LP is computationally efficient but often results in weak cuts. In contrast, \citealt{yangyang2025} characterize the set of optimal dual solutions exactly using a convex program. They propose objective coefficients that lead to cuts that are both tight and Pareto-optimal in the lifted space.

In this work, we propose an alternative approach to mitigating weak cuts resulting from dual degeneracy in methods designed for mixed-integer state variables. Our approach is based on normalization of the ReLU dual. This approach was introduced by \citealt{fullnernew} for the Lagrangian dual obtained by relaxing copy constraints $z_n = x_{a(n)}$ in MIP subproblems. In the normalized dual, additional constraints are imposed on the dual variables, which are shown to address the weakness of the original Lagrangian cuts arising from dual degeneracy. However, the original Lagrangian cuts, based on relaxing original copy constraints $z_n = x_{a(n)}$, only guarantee convergence when the state variables are binary (\citealt{zou2019stochastic}). We extend the concept of normalization to the ReLU-based dual setting to ensure convergence with mixed-integer state variables and to obtain strong cuts.

We focus on the ReLU dual rather than the lifted dual of \citealt{yangyang2025} because the latter introduces binary variables directly into the state space. This can lead to a significant expansion of the state space, substantially increasing the computational effort required to solve the lifted Lagrangian dual and generate the corresponding cuts. In contrast, the ReLU dual introduces auxiliary binary variables only as local variables to the subproblem, keeping the number of dual variables bounded by twice the state dimension in each iteration. This makes solving the dual significantly more efficient than the lifted approach. In our computational study, we observe that the time to solve the dual and the average number of dual iterations both increase with the state dimension. 
We also study the relationship between normalization and regularization and explain why normalization offers important advantages. In particular, we show that normalization can yield tight, Pareto-optimal cuts. To this end, we introduce a notion of Pareto-optimality defined in the original state space for non-linear ReLU cuts and prove that normalization yields Pareto-optimal cuts under this definition. Importantly, normalization provides additional flexibility: cuts need not be both tight and Pareto-optimal simultaneously. By appropriately choosing the normalization coefficients, one may enforce both properties, but this is not required. In contrast, the enhancement strategy proposed by \citealt{yangyang2025} always produces cuts that are both tight and Pareto-optimal. This lack of flexibility can lead to weaker overall approximations of the value function, a behavior that we also illustrate in our computational experiments.

The main contributions of this paper are as follows:
\begin{enumerate}
    \item We extend the normalization framework of \citealt{fullnernew} to the ReLU-based Lagrangian dual of \citealt{deng2024relu}. This extension ensures asymptotic convergence for multistage stochastic integer programs with mixed-integer state variables while producing strong cuts.

    \item We show that normalization resolves the issue of weak cuts resulting from degeneracy in the ReLU dual and identifies solutions that yield strong cuts. The strength of the cuts is established from two perspectives:
    \begin{enumerate}[label=(\alph*)]
        \item We introduce a notion of Pareto-optimality defined in the original state space for nonlinear ReLU cuts, and prove that normalization produces Pareto-optimal cuts under this definition.
        \item We prove that there exists a choice of normalization coefficients that yields tight cuts at the current incumbent. Our existence result also provides insights into selecting these coefficients to obtain tight cuts.
    \end{enumerate}

    \item We analyze the relationship between normalization and regularization for obtaining strong cuts. We show that any cut obtained by regularization can also be obtained by normalization, though the converse is not true. While we establish this relationship in the context of the ReLU dual, it has broader implications. For classical Benders cuts, previous work (\citealt{brandenberg2021refined, hosseini2021deepest}) shows that both normalization and the regularization framework of \citealt{magnanti1981accelerating} can attain Pareto-optimal cuts. However, when multiple such cuts exist, it is unclear whether normalization can attain the same cut as regularization. Our result resolves this question and proves a stronger claim: with an appropriate choice of normalization constraints, the coefficients of any Pareto-optimal cut obtained via a regularization method are also optimal in the normalization dual.

    % \item We introduce valid inequalities that strengthen the mixed-integer representation of ReLU cuts.

    \item Finally, we provide a computational comparison of normalization and regularization approaches, showing that normalization consistently yields stronger cuts, while computational times improve for harder instances that cannot be effectively approximated by the weaker Benders cuts. We also make our code publicly available. To the best of our knowledge, it is the first open-source implementation of a cut-generation method with (asymptotic) convergence guarantees for general MSIPs. The code supports strong cut generation via both normalization and regularization, among other enhancements such as the alternating cut strategy of \citealt{angulo2016improving}.
\end{enumerate}

The remainder of this paper is organized as follows. In Section \ref{sec:relu-norm}, we introduce normalization of the ReLU dual and prove that the resulting cuts are tight and Pareto-optimal. Section~\ref{sec:norm-vs-reg} revisits recently proposed regularization-based approaches and discusses their connection to normalization. Section \ref{sec:comp} presents a computational comparison of normalization and regularization across two classes of problems. Finally, Section~\ref{sec:conclusion} summarizes our contributions and discusses directions for future work.

\section{Normalization of the Dual Formulation} \label{sec:relu-norm}

\subsection{ReLU Dual and ReLU Cuts} \label{sec:relu_dual_cuts}

We first review the ReLU dual and the associated ReLU cuts introduced by \citealt{deng2024relu}. Given an incumbent solution \(\hat{x}_{a(n)}\), cuts are obtained by relaxing the copy constraints \eqref{sub_reform:copy}, which leads to the following Lagrangian relaxation:
\begin{align} \label{relu:lagrn_relax}
      \mc{L}^R_n(\pi_n^+, \pi_n^-; \hat{x}_{a(n)}) := 
      \min_{z_n \in Z_{a(n)}}  
       \ul{Q}_n(z_n) + \sum_{k \in [d_{a(n)}]} \pi_{nk}^+ (z_{nk} - \hat{x}_{a(n),k})^+ 
       + \sum_{k \in [d_{a(n)}]} \pi_{nk}^- (z_{nk} - \hat{x}_{a(n),k})^- ,
\end{align}
where the approximate value function $\ul{Q}_n$ is obtained by replacing $\epi(Q_m)$ in \eqref{sub_reform2:epi} and 
\eqref{sub_reform:orig_constr} with an approximation $\Psi_m$ iteratively obtained from the cut-generation process.
The corresponding Lagrangian dual problem is given by:
\begin{align} \label{relu:dual}
    \max_{\pi_n^+, \pi_n^- \in \mathbb{R}^{d_{a(n)}}} 
    \mc{L}^R_n(\pi_n^+, \pi_n^-; \hat{x}_{a(n)}).
\end{align}
We refer to this problem as the ReLU dual and use the superscript $R$ to denote expressions associated with this formulation. 

For any choice of dual multipliers \(\pi_n^+, \pi_n^- \in \mathbb{R}^{d_{a(n)}}\), a ReLU cut that is valid for the epigraph $\epi_{Z_{a(n)}}(Q_n)$ is given by
\begin{align} \label{cut:ReLU}
        \theta_n \ge 
        \mc{L}^R_n(\pi_n^+, \pi_n^-; \hat{x}_{a(n)}) 
        - \sum_{k \in [d_{a(n)}]} \pi_{nk}^+ (z_{nk} - \hat{x}_{a(n),k})^+ 
        - \sum_{k \in [d_{a(n)}]} \pi_{nk}^- (z_{nk} - \hat{x}_{a(n),k})^-.
\end{align}
\citealt{deng2024relu} show that strong duality holds for the ReLU dual \eqref{relu:dual} under standard assumptions. Strong duality guarantees that a cut generated at the current incumbent is tight. A key advantage of the ReLU-based copy constraints is that they enable the construction of tight cuts even when the state variables are mixed-integer. As a result, the approach ensures asymptotic convergence for general mixed-integer stochastic programs with mixed-integer state variables.

The authors further discuss a linear reformulation of the ReLU cut \eqref{cut:ReLU}. This reformulation models the ReLU functions \((z_{nk} - x_{a(n),k})^+\) and \((z_{nk} - x_{a(n),k})^-\) implicitly through auxiliary continuous variables $w_{nk}^+, w_{nk}^-$ and binary variables $r_{nk}$. In particular, given known bounds on the state variables $x_{nk} \in [0,B_k]$, the ReLU cut can be expressed as the following MIP formulation:
\begin{subequations}\label{cut:relu_linear}
\begin{align*}
    &\theta_n \ge \mc{L}^R_n(\pi_n^+, \pi_n^-; \hat{x}_{a(n)}) 
    - \sum_{k \in [d_{a(n)}]} \pi_{nk}^+ w_{nk}^+ 
    - \sum_{k \in [d_{a(n)}]} \pi_{nk}^- w_{nk}^-\autotag \label{relu_linear:cut} \\
    &w_{nk}^+ - w_{nk}^- = z_{nk} - \hat{x}_{a(n), k}, 
    \quad &\forall k \in [d_{a(n)}],  \autotag \label{relu_linear:equal} \\
    &0 \le w_{nk}^+ \le (B_k - \hat{x}_{a(n),k}) r_{nk}, 
    \quad &\forall k \in [d_{a(n)}], \autotag \label{relu_linear:plus} \\
    &0 \le w_{nk}^- \le \hat{x}_{a(n),k} (1 - r_{nk}), 
    \quad &\forall k \in [d_{a(n)}], \autotag \label{relu_linear:minus} \\
    &r_n \in \{0,1\}^{d_{a(n)}}. \autotag \label{relu_linear:binary}
\end{align*}
\end{subequations}
This reformulation reveals an important connection between ReLU cuts and the original Lagrangian cuts in a lifted space, which we discuss in the next section.

For the rest of the paper, we make the following basic assumptions. First, we assume that sets $X_n$ and $Y_n$ are compact, and all problem data are rational. This ensures that feasible region $H_n(x_{a(n)}) \cap (X_n \times Y_n)$ is compact for all $x_{a(n)} \in \mb{R}^{d_{a(n)}}$. Moreover, under this assumption, the set $X_n$ can be shifted such that all state variables satisfy $x_{nk} \in [0, B_k]$. The data are assumed to be rational to ensure finite MIP-representations and because the convex hull of any MIP-representable set defined by rational data is a rational polyhedron (\citealt{meyer1974existence}). Second, we assume that the domain of the value function satisfies  $\dom(Q_n) = Z_{a(n)} = \prod_{k} [0, B_k]$. Together with \eqref{sub_reform2:copy_space}, this property guarantees relatively complete recourse, a standard assumption in stochastic programming. This requirement can always be enforced by adding continuous variables to polyhedron $H_n$ and penalizing them in the objective. The set $Z_{a(n)}$ is typically taken as a superset of $X_{a(n)}$ to reduce the computational effort in solving the dual problem, but this choice affects the strength of the resulting cuts. For example, in \citealt{zou2019stochastic}, binary restrictions on state variables are relaxed to the interval $[0, 1]$. We refer the reader to \citealt{fullner2024lipschitz} for a detailed discussion of the various options for choosing $Z_{a(n)}$ and their implications.

\subsection{Connection with the Original Lagrangian Cuts} \label{sec:connection}

To establish the connection between ReLU cuts and the original Lagrangian cuts, we lift the domain $Z_{a(n)}$ of the value function $Q_n$ to a higher-dimensional space. Given an incumbent solution $\hat{x}_{a(n)}$, we define the lifted domain as:
\begin{align} \label{def:lifted_domain}
Z^{lift}_{\hat{x}_{a(n)}} = \br{(w_n^+, w_n^-): \exists z_n \in Z_{a(n)} \text{ and } \exists r_{n} \in \br{0, 1}^{d_{a(n)}} \text{ s.t. } \eqref{relu_linear:equal}-\eqref{relu_linear:minus}}.
\end{align}
This lifted space corresponds to the auxiliary variables $(w_n^+, w_n^-)$ used in the linear reformulation \eqref{cut:relu_linear} of the ReLU cut. On this lifted space, we redefine the recourse function as:
\begin{align} \label{prob:lift_rec}
\un{Q}'_n(w_n^+, w_n^-; \hat{x}_{a(n)}) := \un{Q}_n(\hat{x}_{a(n)} + w_n^+ - w_n^-).
\end{align}
Note that $\un{Q}'_n(\mbf{0}, \mbf{0}; \hat{x}_{a(n)}) = \un{Q}_n(\hat{x}_{a(n)})$, which allows us to reformulate the subproblem $\un{Q}'_n(\mbf{0}, \mbf{0}; \hat{x}_{a(n)})$ with the standard copy constraints:
\begin{align*}
    \ul{Q}'_n(\mbf{0}, \mbf{0}; \hat{x}_{a(n)}) = \min_{w_n^+, w_n^-} \quad &\ul{Q}'_n(w_n^+, w_n^-; \hat{x}_{a(n)}) \\
    \text{s.t.} \quad & (w_n^+, w_n^-) = (\mbf{0}, \mbf{0}), \autotag \label{lifted_sub:copy} \\
    & (w_n^+, w_n^-) \in Z_{\hat{x}_{a(n)}}^{lift}.
\end{align*}
By relaxing the copy constraints \eqref{lifted_sub:copy}, we obtain the following Lagrangian relaxation:
\begin{align*}
    \mc{L}^O_n(\pi_n^+, \pi_n^-; \hat{x}_{a(n)}) := \min_{w_n^+, w_n^-} \quad &\ul{Q}'_n(w_n^+, w_n^-; \hat{x}_{a(n)}) + (\pi_n^+, \pi_n^-) \cdot  \begin{pmatrix} w_n^+ \\ w_n^-\end{pmatrix} \autotag \label{prob:lifted_relax}\\
    \text{s.t.} \quad & (w_n^+, w_n^-) \in Z_{\hat{x}_{a(n)}}^{lift}. 
\end{align*}
The superscript $O$ in $\mc{L}_n^{O}$ denotes the original Lagrangian relaxation obtained by relaxing the copy constraints \eqref{lifted_sub:copy}. The corresponding Lagrangian dual problem is:
\begin{align}\label{dual:lifted_lagrn_00}
    \max_{\pi_n^+, \pi_n^-} \mc{L}^O_n(\pi_n^+, \pi_n^-; \hat{x}_{a(n)})  - (\pi_n^+, \pi_n^-) \cdot  \begin{pmatrix} \mbf{0} \\ \mbf{0} \end{pmatrix}.
\end{align}
For any dual multipliers $(\pi_n^+, \pi_n^-)$, the resulting Lagrangian cut takes the form:
\begin{align} \label{cut:lifted_lagrn_00} 
    \theta_n \geq \mc{L}^O_n(\pi_n^+, \pi_n^-; \hat{x}_{a(n)}) - (\pi_n^+, \pi_n^-) \cdot  \begin{pmatrix} w_n^+ \\ w_n^-\end{pmatrix}.
\end{align}
This transformation establishes the following key relationship:
\begin{proposition}[\citealt{deng2024relu}] \label{prop:lifted_lagrangian}
The ReLU Lagrangian cut \eqref{cut:ReLU}, 
generated at $\hat{x}_{a(n)}$ 
for $\epi_{Z_{a(n)}}\left(\un{Q}_n\right)$, 
corresponds to Lagrangian cut \eqref{cut:lifted_lagrn_00} generated at $\left(\mathbf{0}, \mathbf{0}\right)$ for the lifted set $\epi_{Z^{lift}_{\hat{x}_{a(n)}}} \left(\ul{Q}'_n(\cdot, \cdot; \hat{x}_{a(n)})\right)$. 
Moreover, with optimal dual multipliers obtained by solving the ReLU dual \eqref{relu:dual}, the Lagrangian cut \eqref{cut:lifted_lagrn_00} is tight at $\left(\mathbf{0}, \mathbf{0}\right)$ with respect to the lifted value function $\ul{Q}_n'(\cdot, \cdot; \hat{x}_{a(n)})$.
\end{proposition}
Although this result is stated without proof in the original paper, we provide a proof in Appendix~\ref{app:lifted_lagrangian} for completeness. The main takeaway of Proposition \ref{prop:lifted_lagrangian} is that any optimal solution of the ReLU dual \eqref{relu:dual} is also optimal for the lifted Lagrangian dual \eqref{dual:lifted_lagrn_00} and vice-versa. This means that properties of the original Lagrangian cuts can be applied to the transformed problem $\ul{Q}'_n$ over the lifted domain $Z_{\hat{x}_{a(n)}}^{lift}$.

\subsection{Normalized ReLU Dual and Normalized ReLU cuts}

Proposition \ref{prop:lifted_lagrangian} enables us to apply the normalization procedure of \citealt{fullnernew} to the transformed subproblem $\ul{Q}_n'$ in \eqref{prob:lift_rec}, domain $Z_{\hat{x}_{a(n)}}^{lift}$, and incumbent $(\mbf{0}, \mbf{0})$. More precisely, given incumbent solution $(\mbf{0}, \mbf{0})$ to subproblem $\ul{Q}_n'$ and approximation $\hat{\theta}_n$ of the subproblem cost, we consider the following feasibility version of the subproblem $\ul{Q}_n'(\mbf{0}, \mbf{0}; \hat{x}_{a(n)})$:
\begin{subequations} \label{prob:feas}
\begin{align}
    \ul{v}_n ((\mbf{0}, \mbf{0}), \hat{\theta}_n) := \min \quad &0 \notag \\
    & \ul{Q}_n'(w_n^+, w_n^-; \hat{x}_{a(n)}) \leq \hat{\theta}_n  \label{feas:theta} \\
    & (w_n^+ , w_n^-) = (\mbf{0}, \mbf{0}) \label{feas:copy} \\
    & (w_n^+, w_n^-) \in Z_{\hat{x}_{a(n)}}^{lift}. \notag
\end{align}
\end{subequations}
Next, a normalized dual is obtained by dualizing the constraints \eqref{feas:theta} and \eqref{feas:copy} with the corresponding dual variables $\pi_{n0} \in \mb{R}$ and $\pi_n^+, \pi_n^- \in \mb{R}^{d_{a(n)}}$, respectively, and adding a constraint normalizing the dual variables. In particular, consider the following dual problem:
\begin{subequations}\label{norm:dual}
    \begin{align*}
        \ul{v}_n^{ND}((\mbf{0}, \mbf{0}), \hat{\theta}_n) := \max_{\pi_n^+, \pi_n^-, \pi_{n0}} &\left(\mc{L}_n(\pi_n^+, \pi_n^-, \pi_{n0}; \hat{x}_{a(n)}) - (\pi_n^+, \pi_n^-) \pmat{ \mbf{0} \\ \mbf{0}} - \pi_{n0} \hat{\theta}_{n} \right), \autotag \label{norm:obj}\\
        & g_n(\pi_n^+, \pi_n^-, \pi_{n0})  \leq 1, \pi_{n0} \geq 0, \autotag \label{norm:constr}
    \end{align*}
\end{subequations}
where function $g_n(\pi_n^+, \pi_n^-, \pi_{n0})$ in the normalization constraint is of the form $u_n^+ \pi_n^+ + u_n^- \pi_n^- + u_{n0} \pi_{n0}$, where  $u_n^+, u_n^-, u_{n0}$ are given normalization coefficients, and the Lagrangian relaxation ${\mathcal{L}}_n(\pi_n^+, \pi_n^-, \pi_{n0}; \hat{x}_{a(n)})$ is given as follows:
\begin{align*} %\label{norm:relax}
    \mc{L}_n(\pi_n^+, \pi_n^-, \pi_{n0}; \hat{x}_{a(n)}) = \min_{w_n^+, w_n^-} \br{ (\pi_n^+, \pi_n^-) \pmat{w_n^+ \\ w_n^-} + \pi_{n0}\left(\ul{Q}'_n(w_n^+, w_n^-; \hat{x}_{a(n)})\right), (w_n^+, w_n^-) \in Z^{lift}_{\hat{x}_{a(n)}}}.
\end{align*}
In Section~\ref{sec:cut_properties}, we discuss how to choose the normalization coefficients $u_n^+, u_n^-, u_{n0}$ to obtain cuts with desired properties. 

Based on an optimal solution $(\hat{\pi}_n^+, \hat{\pi}_n^-, \hat{\pi}_{n0})$ of the normalized dual \eqref{norm:dual}, the resulting normalized cut is 
\begin{align} \label{norm:cut}
    \hat{\pi}_{n0} \theta_n \geq \mc{L}_n(\hat{\pi}_n^+, \hat{\pi}_n^-, \hat{\pi}_{n0}; \hat{x}_{a(n)}) - (\hat{\pi}_n^+, \hat{\pi}_n^-)\pmat{w_n^+ \\ w_n^-}.
\end{align}
The validity of cut \eqref{norm:cut} for $\epi_{Z_{\hat{x}_{a(n)}}^{lift}}(\ul{Q}_n')$ follows from Lemma 3.7 of \citealt{fullnernew}. It is easy to see that this validity implies that the associated ReLU cut 
\begin{align} \label{cut:lagrn_relu}
    \hat{\pi}_{n0} \theta_n \geq \mc{L}_n(\hat{\pi}_n^+, \hat{\pi}_n^-, \hat{\pi}_{n0}; \hat{x}_{a(n)})
    - \sum_{k \in [d_{a(n)}]} \pi_{nk}^+ (z_{nk} - \hat{x}_{a(n),k})^+
    - \sum_{k \in [d_{a(n)}]} \pi_{nk}^- (z_{nk} - \hat{x}_{a(n),k})^-,
\end{align}
is also valid for the set $\epi_{Z_{a(n)}}(\ul{Q}_n)$. Now, consider the incumbent $(\hat{x}_{a(n)}, \hat{\theta}_n)$ in the original space such that $\hat{\theta}_n < \ul{Q}_n(\hat{x}_{a(n)}) = \ul{Q}_n'(\mbf{0}, \mbf{0}; \hat{x}_{a(n)})$. Then, we seek a cut of the form \eqref{norm:cut} to cut off this incumbent. Our next result establishes that such a cut always exists.

\begin{proposition}\label{prop:norm_cut_existence}
    If $\hat{\theta}_n \geq \ul{Q}_n(\hat{x}_{a(n)}) = \ul{Q}_n'(\mbf{0}, \mbf{0}; \hat{x}_{a(n)})$, then $\ul{v}_n^{ND}((\mbf{0}, \mbf{0}), \hat{\theta}_n) = 0$. Otherwise, there exists a cut of the form \eqref{norm:cut}, such that the incumbent $((\mbf{0}, \mbf{0}), \hat{\theta}_n)$ violates this cut. In the original space, the incumbent $(\hat{x}_{a(n)}, \hat{\theta}_n)$, violates the corresponding ReLU cut of the form \eqref{cut:lagrn_relu}.
\end{proposition}

\begin{proof}
We prove the two statements separately. First, we suppose that $\hat{\theta}_n \geq \ul{Q}'_n(\mbf{0}, \mbf{0}; \hat{x}_{a(n)})$ and show that $\ul{v}_n^{ND}((\mbf{0}, \mbf{0}), \hat{\theta}_n) = 0$. Let $\ol{\co}(f)$ denote the closed convex envelop of function $f$ and $\ol{\co}(f(y))$ denote the function $\ol{\co}(f)$ evaluated at $y$.

Since the closed convex envelope satisfies $\ol{\co}\!\left(\ul{Q}'_n((\mbf{0}, \mbf{0}); \hat{x}_{a(n)})\right) \leq \ul{Q}'_n(\mbf{0}, \mbf{0}; \hat{x}_{a(n)})$, we have $((\mbf{0}, \mbf{0}), \hat{\theta}_n) \in \epi_{Z^{lift}_{\hat{x}_{a(n)}}}\!\left(\ol{\co}\!\left(\ul{Q}'_n(\cdot, \cdot; \hat{x}_{a(n)})\right)\right)$.

Now, note that the normalization function $g_n$ of the form $u_n^+ \pi_n^+ + u_n^- \pi_n^- + u_{n0} \pi_{n0}$ is homogeneous. Therefore, by Lemma 3.14 of \citealt{fullnernew}, we obtain $\ul{v}_n^{ND}((\mbf{0}, \mbf{0}), \hat{\theta}_n) = 0$.  

Next, suppose that $\hat{\theta}_n < \ul{Q}_n(\hat{x}_{a(n)}) = \ul{Q}'_n(\mbf{0}, \mbf{0}; \hat{x}_{a(n)})$.
By Theorem 3.5 and Remark 3.13 of \citealt{fullnernew}, it suffices to show that
$\hat{\theta}_n < \ol{\co}\!\left(\ul{Q}'_n((\mbf{0}, \mbf{0}); \hat{x}_{a(n)})\right)$,
because then there exists a normalized Lagrangian cut of the form \eqref{norm:cut} that separates the incumbent.

To this end, Theorem 3.13 of \citealt{fullner2024lipschitz} implies that
$\ol{\co}\!\left(\ul{Q}'_n((\mbf{0}, \mbf{0}); \hat{x}_{a(n)})\right) = \ul{Q}'_n(\mbf{0}, \mbf{0}; \hat{x}_{a(n)})$
whenever $(\mbf{0}, \mbf{0})$ is an extreme point of $Z^{lift}_{\hat{x}_{a(n)}}$.
We now verify that $(\mbf{0}, \mbf{0})$ is an extreme point. Let $\lambda \in (0,1)$ and let $(v_1^+, v_1^-), (v_2^+, v_2^-) \in Z^{lift}_{\hat{x}_{a(n)}}$ satisfy
$\lambda (v_1^+, v_1^-) + (1-\lambda)(v_2^+, v_2^-) = (\mbf{0}, \mbf{0})$.
With known bounds on state variable $x_{nk} \in [0, B_k], k \in [d_{a(n)}]$, the domain $Z^{lift}_{\hat{x}_{a(n)}} \subseteq \mathbb{R}_+^{d_{a(n)}} \times \mathbb{R}_+^{d_{a(n)}}$, so all components of $v_1^+, v_1^-, v_2^+, v_2^-$ are nonnegative. The above convex-combination equality therefore forces $(v_1^+, v_1^-) = (v_2^+, v_2^-) = (\mbf{0}, \mbf{0})$ proving that $(\mbf{0}, \mbf{0})$ is an extreme point of $Z^{lift}_{\hat{x}_{a(n)}}$.
Consequently,
$\ol{\co}\!\left(\ul{Q}'_n((\mbf{0}, \mbf{0}); \hat{x}_{a(n)})\right) = \ul{Q}'_n(\mbf{0}, \mbf{0}; \hat{x}_{a(n)})$
and since $\hat{\theta}_n$ is strictly smaller than this value, a separating normalized cut exists.

Using linear reformulation in \eqref{relu_linear:equal}-\eqref{relu_linear:binary}, this translates to a ReLU cut of the form \eqref{cut:lagrn_relu} that cuts-off the incumbent $(\hat{x}_{a(n)}, \hat{\theta}_n)$.
\end{proof}

\ignore{We present the proof in Appendix~\ref{app:proof_norm_cut_existence}.} 

 In \citealt{fullnernew}, the existence of a separating normalized Lagrangian cut is established under the condition $(\hat{x}_{a(n)}, \hat{\theta}_n) \notin \epi_{Z_{a(n)}}\!\left(\ol{\co}(\ul{Q}_n)\right)$. Proposition \ref{prop:norm_cut_existence} strengthens this by showing that, a normalized cut of the form \eqref{cut:lagrn_relu} exists whenever $(\hat{x}_{a(n)}, \hat{\theta}_n) \notin \epi_{Z_{a(n)}}(\ul{Q}_n)$. Consequently, normalized ReLU cuts preserve the separation property needed for asymptotic convergence.

\subsection{Normalization Coefficients and their Impact on Cut Quality}\label{sec:cut_properties}

In this section, we discuss how to choose the normalization coefficients $u_n^+, u_n^-$, and $u_{n0}$ in the normalization constraint \eqref{norm:constr}. This choice impacts cut quality, which we evaluate through two properties: Pareto-optimality and tightness at the incumbent. We first study these properties in the extended space and then translate the results back to the original state space.

\subsubsection{Pareto-optimal Cuts}

The concept of Pareto-optimal cuts is introduced in \citealt{magnanti1981accelerating} for affine cuts.
\begin{definition}[Pareto-optimal affine cut]
\label{def:pareto_optimal}
A cut of the form $\theta_n \ge \ell^1 - (\pi^1)^{\top} x_{a(n)}$ dominates the cut $\theta_n \ge \ell^2 - (\pi^2)^{\top} x_{a(n)}$ if
$\ell^1 - (\pi^1)^{\top} x_{a(n)} \ge \ell^2 - (\pi^2)^{\top} x_{a(n)}$ for all $x_{a(n)} \in X_{a(n)}$, with strict inequality for at least one point in $X_{a(n)}$.
A valid cut is Pareto-optimal for reference set $\epi_{X_{a(n)}}(\ul{Q}_n)$ if no other valid cut dominates it.
\end{definition}

Pareto-optimality is always defined relative to a reference set. Importantly, while establishing Pareto-optimality on a larger set such as $\epi_{\conv(X_{a(n)})}(\ol{\co}(\ul{Q}_n))$ may be easier, it does not guarantee Pareto-optimality on the original epigraph $\epi_{X_{a(n)}}(\ul{Q}_n)$. We refer the reader to Figure 3.4 and Figure 3.5 of \citealt{stursberg2019mathematics} for an example depicting the importance of the reference set in the definition of the Pareto-optimal cuts.

We now apply this concept to the normalized (and linear) cut \eqref{norm:cut} in the extended space. Our goal is to identify the choice of normalization coefficients $u_n^+, u_n^-, u_{n0}$ in constraint \eqref{norm:constr} that yields Pareto-optimal cuts. The following result provides this characterization. Throughout, $\relint(\cdot)$ denotes the relative interior of a set.
\begin{proposition} \label{prop:fullner_pareto}
For all $(u_n^+, u_n^-, u_{n0}) \in \relint\cbr{\epi(\ol{\co}(\ul{Q}_n'(\cdot, \cdot; \hat{x}_{a(n)}))) - (\mbf{0}, \mbf{0}, \hat{\theta}_n)}$, any optimal solution $(\hat{\pi}_n^+, \hat{\pi}_n^-, \hat{\pi}_{n0})$ to the normalized dual \eqref{norm:dual} with $\hat{\pi}_{n0} > 0$ defines a Pareto-optimal cut of form \eqref{norm:cut} for the reference set $\epi(\ol{\co}(\ul{Q}_n'(\cdot, \cdot; \hat{x}_{a(n)})))$ on $\conv(Z_{\hat{x}_{a(n)}}^{lift})$. 
\end{proposition}
The proof follows immediately from Theorem 3.25 of \citealt{fullnernew}. In the standard literature, normalization coefficients that yield Pareto-optimal cuts or satisfy the requirement in Proposition \ref{prop:fullner_pareto} are referred to as \emph{core points}. While identifying core points remains a challenge and typically requires solving an additional MIP subproblem (\citealt{magnanti1981accelerating,fullnernew,yangyang2025}), we show that when the incumbent solution $\hat{x}_{a(n)}$ satisfies $0 < \hat{x}_{a(n), k} < B_k$ for all $k \in [d_{a(n)}]$, a core point can be constructed efficiently without additional subproblem solves, as shown below. 

\ignore{The proof is deferred to Appendix \ref{app:proof_of_core_point}.}

\begin{proposition} \label{prop:core_point}
    Suppose that the incumbent solution $\hat{x}_{a(n)}$ satisfies $0 < \hat{x}_{a(n), k} < B_k$ for all $k \in [d_{a(n)}]$. Then any point $(u_n^+, u_n^-, u_{n0})$ with 
    \begin{equation} \label{eq:core_point_coefficients}
        (u_{nk}^+ , u_{nk}^-=u_{nk}^+) \in \relint\left(\conv\Bigl\{(0,0),\ \bigl(B_k-\hat{x}_{a(n),k},\,0\bigr),\ \bigl(0,\,\hat{x}_{a(n),k}\bigr)\Bigr\}\right)
    \end{equation}
    for all $k \in [d_{a(n)}]$,
    and $u_{n0} \in (\ul{Q}_n(\hat{x}_{a(n)}) - \hat{\theta}_n,  \infty)$, defines a core point, that is, $(u_n^+, u_n^-, u_{n0}) \in \relint\cbr{\epi(\ol{\co}(\ul{Q}_n'(\cdot, \cdot; \hat{x}_{a(n)}))) - (\mbf{0}, \mbf{0}, \hat{\theta}_n)}$.
\end{proposition}

\begin{proof}
Consider any coefficients $(u_n^+, u_n^-, u_{n0})$ satisfying condition \eqref{eq:core_point_coefficients} with $u_{n0} \in (\ul{Q}_n(\hat{x}_{a(n)}) - \hat{\theta}_n,  \infty)$. We prove that $(u_n^+, u_n^-, u_{n0}) \in \relint\cbr{\epi(\ol{\co}(\ul{Q}_n'(\cdot, \cdot; \hat{x}_{a(n)}))) - (\mbf{0}, \mbf{0}, \hat{\theta}_n)}$.

Since $\ul{Q}_n'(\cdot, \cdot; \hat{x}_{a(n)})$ is defined on $Z_{\hat{x}_{a(n)}}^{lift}$, $\ol{\co}(\ul{Q}_n'(\cdot, \cdot; \hat{x}_{a(n)}))$ is defined on $\conv(Z_{\hat{x}_{a(n)}}^{lift})$. Thus, it suffices to show that $(u_n^+, u_n^-) \in \relint(\conv(Z_{\hat{x}_{a(n)}}^{lift}))$ and $u_{n0} > \ol{\co}(\ul{Q}_n'(u_n^+, u_n^-; \hat{x}_{a(n)})) - \hat{\theta}_n$.

We first show that the chosen $u_{n}^+, u_{n}^-$ satisfy the relative interior requirement. To this end, we analyze the lifted domain $Z_{\hat{x}_{a(n)}}^{lift}$ defined in \eqref{def:lifted_domain}. For notational simplicity, we denote $Z^L := Z_{\hat{x}_{a(n)}}^{lift}$ and $Z := Z_{a(n)}$ in this proof. Since $Z = \prod_{k} [0, B_k]$ under the assumptions made in Section \ref{sec:relu_dual_cuts}, we can decompose $Z^L$ dimension-wise. Define $Z^L_k$ for $k \in [d_{a(n)}]$ as follows:
\begin{align*}
    Z^L_k = \{(w_{nk}^+, w_{nk}^-): {} & \exists z_{nk} \in [0, B_k] \text{ and } \exists r_{nk} \in \br{0,1} \text{ s.t. } \\
    & w_{nk}^+ - w_{nk}^- = z_{nk} - \hat{x}_{a(n), k}, \\
    & 0 \le w_{nk}^+ \le (B_k - \hat{x}_{a(n),k}) r_{nk}, \\
    & 0 \le w_{nk}^- \le \hat{x}_{a(n),k} (1 - r_{nk})\}.
\end{align*}
Now, we can write $Z^L = \prod_{k \in [d_{a(n)}]} Z^L_k$. This implies that
\begin{align}\label{proof:zhull_decom}
    \conv(Z^L) = \prod_{k \in [d_{a(n)}]} \conv(Z^L_k).
\end{align}
This further implies that $\relint(\conv(Z^L)) = \prod_{k \in [d_{a(n)}]} \relint(\conv(Z^L_k))$. Now, it is easy to see that $\conv(Z_k^L)$ is the convex hull of points $(0,0), (B_k - \hat{x}_{nk}, 0)$ and $(0, \hat{x}_{nk})$. The given $(u_{nk}^+, u_{nk}^-)$ in \eqref{eq:core_point_coefficients} thus belongs to $\relint(\conv(Z_k^L))$. Consequently, the entire vector $(u_{nk}^+, u_{nk}^-)_{k \in [d_{a(n)}]}$ defined in \eqref{eq:core_point_coefficients} belongs to $\relint(\conv(Z^L))$.

Next, we establish the second requirement: $u_{n0} > \ol{\co}(\ul{Q}_n'(u_n^+, u_n^-; \hat{x}_{a(n)})) - \hat{\theta}_n$. Since $\ul{Q}_n'(\mbf{0}, \mbf{0}; \hat{x}_{a(n)}) = \ul{Q}_n(\hat{x}_{a(n)})$ and the closed convex envelope satisfies $\ol{\co}(\ul{Q}_n'(u_n^+, u_n^-; \hat{x}_{a(n)})) \le \ul{Q}_n'(u_n^+, u_n^-; \hat{x}_{a(n)})$, we have
\begin{align*}
\ol{\co}(\ul{Q}_n'(u_n^+, u_n^-; \hat{x}_{a(n)})) - \hat{\theta}_n &\le \ul{Q}_n'(u_n^+, u_n^-; \hat{x}_{a(n)}) - \hat{\theta}_n \\
&= \ul{Q}_n(\hat{x}_{a(n)} + u_n^+ - u_n^-) - \hat{\theta}_n.
\end{align*}
The second equality follows from the definition of $\ul{Q}_n'$ in \eqref{prob:lift_rec}. Since $u_n^+ = u_n^-$ componentwise, we have $\hat{x}_{a(n)} + u_n^+ - u_n^- = \hat{x}_{a(n)}$, so $\ul{Q}_n(\hat{x}_{a(n)} + u_n^+ - u_n^-) = \ul{Q}_n(\hat{x}_{a(n)})$. Therefore, $\ol{\co}(\ul{Q}_n'(u_n^+, u_n^-; \hat{x}_{a(n)})) - \hat{\theta}_n \le \ul{Q}_n(\hat{x}_{a(n)}) - \hat{\theta}_n$. Since $u_{n0} \in (\ul{Q}_n(\hat{x}_{a(n)}) - \hat{\theta}_n, \infty)$, we conclude that $u_{n0} > \ol{\co}(\ul{Q}_n'(u_n^+, u_n^-; \hat{x}_{a(n)})) - \hat{\theta}_n$, completing the proof.
\end{proof}

The advantage of Proposition \ref{prop:core_point} is that it provides a computationally efficient way to obtain a core point when the incumbent solution does not lie on the boundary. Specifically, the scalar $u_{n0}$ can be computed directly from $\ul{Q}_n(\hat{x}_{a(n)})$, which is readily available, without requiring additional subproblem solves. For incumbent solutions that lie on the boundary of the feasible region, we discuss our approach for deriving a core point in the computational results section.

Having established Pareto-optimality of the normalized Lagrangian cuts in the extended space (Proposition \ref{prop:fullner_pareto}), we now investigate the implications of this property in the original space. Towards this end, we consider the following definition of Pareto-optimal cuts.

\begin{definition}[Pareto-optimal $h$-cut]
\label{def:pareto_optimal_hcut}
A cut of the form $\theta_n \geq h(x_{a(n)}; \ell^1, \pi^1)$ dominates the cut $\theta_n \geq h(x_{a(n)}; \ell^2, \pi^2)$ if $h(x_{a(n)}; \ell^1, \pi^1) \geq h(x_{a(n)}; \ell^2, \pi^2)$ for all $x_{a(n)} \in X_{a(n)}$, with strict inequality for at least one point in $X_{a(n)}$. A valid $h$-cut is Pareto-optimal for reference set $\epi_{X_{a(n)}}(\ul{Q}_n)$ if no other valid $h$-cut dominates it.
\end{definition}

Note that Definition \ref{def:pareto_optimal_hcut} generalizes Definition \ref{def:pareto_optimal} of Pareto-optimal linear cuts. For a given incumbent $\hat{x}_{a(n)}$, the function $h$ of interest is:
\begin{align} \label{eq:h_function}
    h(x_{a(n)}; \pi_n^+, \pi_n^-, \pi_{n0}) &= \frac{1}{\pi_{n0}} \Bigg(\mc{L}_n(\pi_n^+, \pi_n^-, \pi_{n0}; \hat{x}_{a(n)})   \nonumber \\
    &\quad - \sum_{k \in [d_{a(n)}]} \pi_{nk}^+ (x_{a(n),k} - \hat{x}_{a(n),k})^+ - \sum_{k \in [d_{a(n)}]} \pi_{nk}^- (x_{a(n),k} - \hat{x}_{a(n),k})^- \Bigg).
\end{align}
The cut $\theta_n \geq h(x_{a(n)}; \pi_n^+, \pi_n^-, \pi_{n0})$ is equivalent to the ReLU cut \eqref{cut:lagrn_relu} in the original space. This formulation allows us to define Pareto-optimality for all valid ReLU cuts at a given incumbent. 

In defining Pareto optimality, it is essential to specify the reference set over which Pareto optimality is attained. In our setting, the reference set of interest is defined as follows:
\begin{align*} 
    \mc{H}_n = \{(x_{a(n)}, \theta_n): &x_{a(n)} \in \conv(Z_{a(n)}), \\
    &\theta_n \geq h(x_{a(n)}; \pi_n^+, \pi_n^-, \pi_{n0}), \forall \pi_n^+, \pi_n^- \in \mb{R}^{d_{a(n)}}, \pi_{n0} \geq 0\}. \autotag \label{def:ref_set}
\end{align*}

Both function $h$ and set $\mc{H}_n$ depend on the incumbent $\hat{x}_{a(n)}$, although we omit this dependence from the notation for brevity. The set $\mc{H}_n$ can be interpreted as the epigraph generated by adding all ReLU cuts of the form \eqref{cut:lagrn_relu} for a fixed incumbent $\hat{x}_{a(n)}$. We depict an example of the set $\mc{H}_n$ in Figure \ref{fig:h_example}.

\begin{figure}[ht]
\centering
% Requires: \usepackage{tikz} \usetikzlibrary{patterns,arrows.meta}
\begin{tikzpicture}[x=1.35cm,y=1.0cm,>=Stealth]

% --- axes/frame ---
\draw[thick] (0,0) rectangle (4,5);
\node[below,yshift=-11pt] at (2,0) {$x_{a(n)}$};
\node[left, xshift=-15pt]  at (0,2.5) {$\theta_n$};

% ticks (optional)
\foreach \xx in {0,1,2,3,4} \draw (\xx,0) -- (\xx,-0.08) node[below] {\small \xx};
\foreach \yy in {0,1,2,3,4,5} \draw (0,\yy) -- (-0.08,\yy) node[left] {\small \yy};

% --- epigraph of h(x)+2 (patterned region above the dashed black curve) ---
% Use pattern + light gray so it survives grayscale + is printer-friendly.
\fill[pattern=crosshatch, pattern color=black!55, fill=black!10]
  (0,5) -- (4,5) -- (4,4)
  -- (2,2) -- (1,3) -- (0,2) -- cycle;

% --- curves: f(x)+2 (solid gray) and h(x)+2 (dashed black) ---
\draw[black!40,ultra thick] (0,2) -- (1,3) -- (2,2) -- (3,4) -- (4,4);
%\draw[black,very thick,dash pattern=on 3.2pt off 2.2pt] (0,2) -- (1,3) -- (2,2) -- (4,4);

% --- overlay lines (distinct dash patterns; no color) ---
% theta = 3 - |x-1|
\draw[black,thick,dash pattern=on 1.2pt off 1.8pt] (0,2) -- (1,3) -- (4,0);

% theta = 1 + |x-1|
\draw[black,thick,dash pattern=on 6pt off 2pt on 1.5pt off 2pt] (0,2) -- (1,1) -- (4,4);

% --- legend (styles match plot) ---
\begin{scope}[shift={(4.25,4.6)}]
  \draw[black!40,ultra thick] (0,0) -- (0.8,0) node[right,black] {$\ul{Q}_n(x_{a(n)})$};

  \draw[black,thick,dash pattern=on 1.2pt off 1.8pt] (0,-0.45) -- (0.8,-0.45)
    node[right] {$\theta_n = 3-(x_{a(n)}-1)^+ - (x_{a(n)}-1)^-$};

  \draw[black,thick,dash pattern=on 6pt off 2pt on 1.5pt off 2pt] (0,-0.9) -- (0.8,-0.9)
    node[right] {$\theta_n = 1+(x_{a(n)}-1)^+ + (x_{a(n)}-1)^- $};

  % optional legend key for the epigraph region
  \begin{scope}[yshift=-1.35cm]
    \draw[black,thin] (0,0) rectangle (0.8,0.22);
    \fill[pattern=crosshatch, pattern color=black!55, fill=black!10]
      (0,0) rectangle (0.8,0.22);
    \node[right] at (0.9,0.11) {$\mc{H}_n$};
  \end{scope}
\end{scope}

\end{tikzpicture}
\caption{The solid (grey) line depicts $\ul{Q}_n(x_{a(n)})$, a piecewise-linear function. The domain $Z_{a(n)} = [0,4]$ and the shaded region shows the epigraph $\mc{H}_n$ at incumbent $\hat{x}_{a(n)}=1$. Two ReLU cuts are shown: $\theta_n = 3-(x_{a(n)}-1)^+ - (x_{a(n)}-1)^-$ (dotted line) and $\theta_n = 1+(x_{a(n)}-1)^+ + (x_{a(n)}-1)^-$ (dashed line). These cuts are obtained using different normalization coefficients, and their epigraph intersection equals $\mc{H}_n$. Observe that $\epi(\ul{Q}_n)\subsetneq \mc{H}_n \subsetneq \epi(\overline{\co}(\ul{Q}_n))$.}
\label{fig:h_example}
\end{figure}
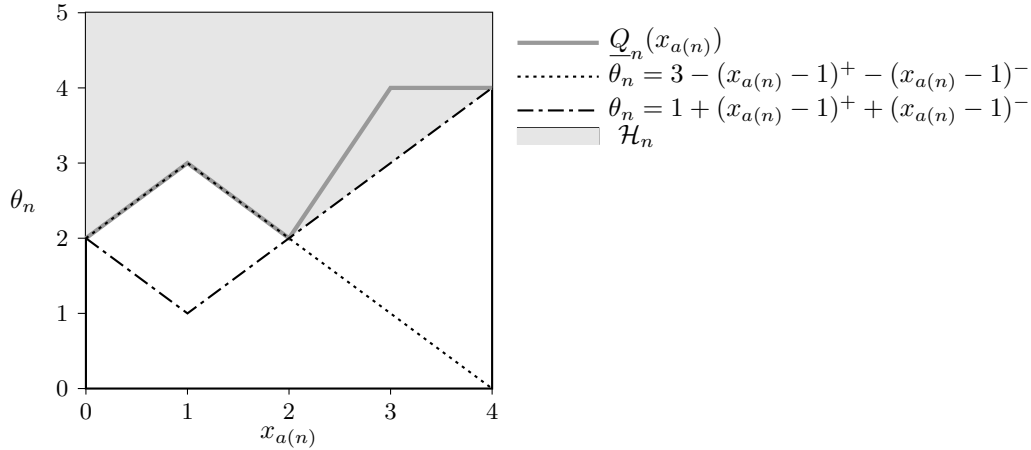

We now establish a useful property of the proposed cut with respect to this new notion of Pareto optimality.
\begin{proposition} \label{prop:relu_pareto}
    Let $\hat{x}_{a(n)}$ be the incumbent solution and function $h$ be as defined in \eqref{eq:h_function}. For all points $(u_n^+, u_n^-, u_{n0})$ that belong to $\relint\cbr{\epi(\ol{\co}(\ul{Q}_n'(\cdot, \cdot; \hat{x}_{a(n)}))) - (\mbf{0}, \mbf{0}, \hat{\theta}_n)}$, any optimal solution $(\hat{\pi}_n^+, \hat{\pi}_n^-, \hat{\pi}_{n0})$ to the normalized dual \eqref{norm:dual} with $\hat{\pi}_{n0} > 0$ defines a Pareto-optimal $h$-cut of form $\theta_n \geq h(x_{a(n)}; \hat{\pi}_n^+, \hat{\pi}_n^-, \hat{\pi}_{n0})$ for reference set $\mc{H}_n$. Further, we have $\mc{H}_n \subseteq \epi_{\conv(Z_{a(n)})}(\ol{\co}(\ul{Q}_n))$, and if $\ol{\co}(\ul{Q}_n)(\hat{x}_{a(n)}) < \ul{Q}_n(\hat{x}_{a(n)})$, then $\mc{H}_n \subsetneq \epi_{\conv(Z_{a(n)})}(\ol{\co}(\ul{Q}_n))$.
\end{proposition}

\begin{proof}
    Consider the dual problem \eqref{norm:dual} with normalization coefficients 
\begin{align*}
    (u_n^+, u_n^-, u_{n0}) \in \relint\cbr{\epi(\ol{\co}(\ul{Q}_n'(\cdot, \cdot; \hat{x}_{a(n)}))) - (\mbf{0}, \mbf{0}, \hat{\theta}_n)}.
\end{align*} 
Let $(\hat{\pi}_n^+, \hat{\pi}_n^-, \hat{\pi}_{n0})$ be an optimal solution of the dual problem with $\pi_{n0} > 0$. Assume, for a contradiction, that the resulting ReLU cut,
\begin{align}\label{proof:h_opt}
    \theta_n \geq h(x_{a(n)}; \hat{\pi}_n^+, \hat{\pi}_n^-, \hat{\pi}_{n0})
\end{align}
is not a Pareto-optimal $h$-cut for $\mc{H}_n$. This implies that there exists different set of dual variables $(\ti{\pi}_n^+, \ti{\pi}_n^-, \ti{\pi}_{n0})$ with $\ti{\pi}_{n0} > 0$, such that the cut
\begin{align} \label{proof:h_alt}
    \theta_n \geq h(x_{a(n)}; \ti{\pi}_n^+, \ti{\pi}_n^-, \ti{\pi}_{n0})
\end{align}
dominates cut \eqref{proof:h_opt} for $(x_{a(n)}, \theta_n) \in \mc{H}_n$. Now, reformulating the ReLU functions, we can represent cuts \eqref{proof:h_opt} and \eqref{proof:h_alt} in the form \eqref{norm:cut}, in the extended space $Z_{\hat{x}_{a(n)}}^{lift}$. Specifically, we reformulate \eqref{proof:h_opt} as:
\begin{align} \label{proof:h_opt_MIP}
     \theta_n \geq \frac{1}{\hat{\pi}_{n0}} \left(\mc{L}_n(\hat{\pi}_n^+, \hat{\pi}_n^-, \hat{\pi}_{n0}; \hat{x}_{a(n)}) - (\hat{\pi}_n^+, \hat{\pi}_n^-)\pmat{w_n^+ \\ w_n^-} \right).
\end{align}
 Similarly, we reformulate \eqref{proof:h_alt} as:
\begin{align} \label{proof:h_alt_MIP}
    \theta_n \geq \frac{1}{\ti{\pi}_{n0}} \left(\mc{L}_n(\ti{\pi}_n^+, \ti{\pi}_n^-, \ti{\pi}_{n0}; \hat{x}_{a(n)}) - (\ti{\pi}_n^+, \ti{\pi}_n^-)\pmat{w_n^+ \\ w_n^-}\right).
\end{align}
 We show that if cut \eqref{proof:h_alt} dominates \eqref{proof:h_opt} on $\mc{H}_n$, then cut \eqref{proof:h_alt_MIP} dominates cut \eqref{proof:h_opt_MIP} in the extended space $\epi(\ol{\co}(\ul{Q}_n'(\cdot, \cdot; \hat{x}_{a(n)})))$ on $\conv(Z_{\hat{x}_{a(n)}}^{lift})$. This contradicts Proposition \ref{prop:fullner_pareto}, which states that cut \eqref{proof:h_opt_MIP} is Pareto-optimal on $\conv(Z_{\hat{x}_{a(n)}}^{lift})$. 

Since, the set $Z_{a(n)} = \prod_{k \in [d_{a(n)}]} [0, B_k]$, we consider points $x_{a(n)}\in \prod_{k \in [d_{a(n)}]} \br{0, \hat{x}_{a(n),k}, B_k} \subset Z_{a(n)}$, where $\prod_{k \in [d_{a(n)}]} \br{0, \hat{x}_{a(n),k}, B_k}$ denotes the cartesian product of the set $\br{0, \hat{x}_{a(n),k}, B_k}$ across all dimensions $k \in [d_{a(n)}]$. Furthermore, since cut \eqref{proof:h_alt} dominates cut \eqref{proof:h_opt} on $\mc{H}_n$, by definition of $\mc{H}_n$, the dominance will also hold over $\prod_{k \in [d_{a(n)}]} \br{0, \hat{x}_{a(n),k}, B_k}$. In other words:
\begin{align} \label{proof:h_domination}
    h(x_{a(n)}; \ti{\pi}_n^+, \ti{\pi}_n^-, \ti{\pi}_{n0}) \geq h(x_{a(n)}; \hat{\pi}_n^+, \hat{\pi}_n^-, \hat{\pi}_{n0}), \quad \forall x_{a(n)} \in \prod_{k \in [d_{a(n)}]} \br{0, \hat{x}_{a(n),k}, B_k}.
\end{align}

Next, for $i \in [2d_{a(n)}]$, let $e^i \in \mathbb{R}^{2d_{a(n)}}$ denote the $i^{th}$ unit vector, i.e., the vector whose $i^{th}$ component equals $1$ and whose remaining components are zero. We index the coordinates of $\mathbb{R}^{2d_{a(n)}}$ so that, for each $k \in [d_{a(n)}]$, the $k^{th}$ component corresponds to the variable $w_{nk}^+$, while the $(d_{a(n)} + k)^{th}$ component corresponds to the variable $w_{nk}^-$. 

In addition, define scalars $v_i$ for $i \in [2d_{a(n)}]$ by
\begin{align*}
v_k := B_k - \hat{x}_{a(n),k}, \qquad 
v_{d_{a(n)} + k} := \hat{x}_{a(n),k}, \quad \forall k \in [d_{a(n)}].
\end{align*}
Now, consider all points $\br{v_i e^i}, i \in [2d_{a(n)}]$. Clearly, they belong to the set $Z_{\hat{x}_{a(n)}}^{lift}$ and, therefore, also to the set $\conv(Z_{\hat{x}_{a(n)}}^{lift})$. Under the mapping \eqref{relu_linear:equal}, these points belong to $\prod_{k \in [d_{a(n)}]} \br{0, \hat{x}_{a(n),k}, B_k}$ in the $Z_{a(n)}$ space. Furthermore, for these points, the right-hand side of \eqref{proof:h_opt_MIP} and \eqref{proof:h_alt_MIP} match with right-hand side of \eqref{proof:h_opt} and \eqref{proof:h_alt}, respectively, on the corresponding points in $\prod_{k \in [d_{a(n)}]} \br{0, \hat{x}_{a(n),k}, B_k}$. Owing to domination relation in \eqref{proof:h_domination}, this implies that cut \eqref{proof:h_alt_MIP} dominates cut \eqref{proof:h_opt_MIP} on points $\br{v_i e^i}, i \in [2d_{a(n)}]$ in the extended space. Now, due to the decomposition relation established in \eqref{proof:zhull_decom}, we know that $\conv(\br{v_i e^i}_{i \in [2 d_{a(n)}]}) = \conv(Z_{\hat{x}_{a(n)}}^{lift})$. This proves that cut \eqref{proof:h_alt_MIP} dominates cut \eqref{proof:h_opt_MIP} on entire $\conv(Z_{\hat{x}_{a(n)}}^{lift})$. This contradicts Proposition \ref{prop:fullner_pareto}, hence our original claim in Proposition \ref{prop:relu_pareto} is true.

Now, we prove the second part of the proposition which states that set $\mc{H}_n \subseteq \epi_{\conv(Z_{a(n)})}(\ol{\co}(\ul{Q}_n))$ and if $\ol{\co}(\ul{Q}_n)(\hat{x}_{a(n)}) < \ul{Q}_n(\hat{x}_{a(n)})$, then $\mc{H}_n \subsetneq \epi_{\conv(Z_{a(n)})}(\ol{\co}(\ul{Q}_n))$. According to Proposition~4 of \citealt{deng2024relu}, any tight Lagrangian cut derived in the original space, from the Lagrangian dual obtained by relaxing the standard copy constraints, is a ReLU Lagrangian cut \eqref{cut:ReLU} with $\mc{L}^R_n(\pi_n^+, \pi_n^-; \hat{x}_{a(n)}) = \ul{Q}_n(\hat{x}_{a(n)})$. More generally, the same proof works in showing that any normalized Lagrangian cut derived in the original space can be expressed as a normalized ReLU Lagrangian cut. This implies that the set $\mc{H}_n$ defined in \eqref{def:ref_set} also consists of normalized Lagrangian cuts (linear cuts obtained in original space). From Theorem~3.5, Remark~3.13, and Lemma~3.14 of \citealt{fullnernew}, the collection of all normalized Lagrangian cuts recovers the epigraph $\epi_{\conv(Z_{a(n)})}(\ol{\co}(\ul{Q}_n))$. It therefore follows that $\mc{H}_n$ with linear Lagrangian cuts and additional non-linear ReLU cuts is a subset of $\epi_{\conv(Z_{a(n)})}(\ol{\co}(\ul{Q}_n))$. 

Next consider the case when $\ol{\co}(\ul{Q}_n)(\hat{x}_{a(n)}) < \ul{Q}_n(\hat{x}_{a(n)})$. We know that point $(\hat{x}_{a(n)}, \ol{\co}(\ul{Q}_n)(\hat{x}_{a(n)}))$ belongs to  $\epi_{\conv(Z_{a(n)})}(\ol{\co}(\ul{Q}_n))$. Since, $\ol{\co}(\ul{Q}_n)(\hat{x}_{a(n)}) < \ul{Q}_n(\hat{x}_{a(n)})$, we also know that there exists a ReLU cut that violates this point. So, this point does not belong to $\mc{H}_n$. This proves that  $\mc{H}_n \subsetneq \epi_{\conv(Z_{a(n)})}(\ol{\co}(\ul{Q}_n))$.
\end{proof}

Note that if we do not employ ReLU copy constraints and instead use the standard copy constraints, then normalizing the respective dual yields linear Lagrangian cuts. As shown by \citealt{fullnernew}, such cuts can attain Pareto-optimality only with respect to the epigraph
$
\epi_{\conv(Z_{a(n)})}\!\left(\ol{\co}\!\left(\ul{Q}_n\right)\right).
$
Since the set $ \mc{H}_n $ is a strict subset of this epigraph, Proposition~\ref{prop:relu_pareto} establishes that the nonlinear ReLU cuts achieve a stronger notion of Pareto-optimality than the linear cuts derived by normalizing the standard Lagrangian dual.

\subsubsection{Tight Cuts}

In discussing tight cuts, it is important to distinguish between two notions of cut strength that are often conflated in the literature. The first is tightness at incumbent solution: a cut is tight if it matches the cost-to-go function $\ul{Q}_n$ exactly at the current state variable solution. Under this definition, several families of cuts—including ReLU cuts obtained by solving dual \eqref{relu:dual}—are tight. For instance, $\Lambda$-shaped cuts (see \citealt{ahmed2022stochastic}, \citealt{deng2024relu}), which extend classical $L$-shaped cuts to mixed-integer state variables and are obtained by imposing $\pi_n^+ = \pi_n^-$ in \eqref{relu:dual}, are also tight at the incumbent. Tightness is a convenient sufficient condition for convergence—if one generates a tight cut at every iteration, asymptotic convergence follows. It is not, however, a necessary condition. Proposition~\ref{prop:norm_cut_existence} establishes asymptotic convergence without requiring tightness at the incumbent.

The second notion is global approximation quality: a cut is globally strong if it yields strong lower bounds not only at the incumbent but also across the feasible region. $\Lambda$-shaped cuts are typically weak in this sense, since their approximation can be weak away from the incumbent (see computational studies of integer $L$-shaped cuts (\citealt{zou2019stochastic, bansal2024computational}) and of $\Lambda$-shaped cuts (\citealt{deng2024relu}) for reference). For this reason, the Pareto-optimality concept introduced in the previous section provides a more meaningful measure of approximation quality. In this section, we study how to select normalization coefficients to obtain ReLU cuts that are tight at the incumbent while also being Pareto-optimal over the domain.

\begin{proposition} \label{prop:tight_relu}
    There exists an $\epsilon$-ball $B_{\epsilon}(\mbf{0}, \mbf{0})$, around $(\mbf{0}, \mbf{0})$ in extended space $\conv(Z_{\hat{x}_{a(n)}}^{lift})$ such that for normalization coefficients $(u_n^+, u_n^-, u_{n0})$ satisfying:
    \begin{align*}
        &(u_{n}^+, u_n^-) \in B_{\epsilon}(\mbf{0}, \mbf{0}) \cap \relint(\conv(Z^{lift}_{\hat{x}_{a(n)}})), \\
        &u_{n0} \geq  \ul{Q}_n'(u_n^+, u_n^-; \hat{x}_{a(n)}) - \hat{\theta}_n,
    \end{align*}
    an optimal dual solution $(\hat{\pi}_n^+, \hat{\pi}_n^-, \hat{\pi}_{n0})$ to \eqref{norm:dual} with $\hat{\pi}_{n0} > 0$ defines a tight ReLU cut at the incumbent solution $\hat{x}_{a(n)}$, that is, $\frac{1}{\hat{\pi}_{n0}} \left(\mc{L}_n(\hat{\pi}_n^+, \hat{\pi}_n^-, \hat{\pi}_{n0}; \hat{x}_{a(n)}) \right) = \ul{Q}_n(\hat{x}_{a(n)})$. Furthermore, the cut is a Pareto-optimal $h$-cut over the reference set $\mc{H}_n$ for $h$ defined in \eqref{eq:h_function}.
\end{proposition}

\begin{proof}
    We prove this result in four parts.
\begin{enumerate}[label=6.\arabic*]
    \item \label{proof:tight_relu:1} We first show that the function $\ul{Q}_n'(\cdot, \cdot; \hat{x}_{a(n)})$ is piecewise polyhedral with finitely many pieces.
    \item \label{proof:tight_relu:2} There is a neighborhood $B_{\epsilon}(\mbf{0}, \mbf{0})$ of the extreme point $(\mbf{0}, \mbf{0})$ in the lifted space $\conv(Z_{\hat{x}_{a(n)}}^{lift})$ where all affine pieces that locally define the function $\ol{\co}(\ul{Q}_n'(\cdot, \cdot; \hat{x}_{a(n)}))$ must agree at $(\mbf{0}, \mbf{0})$. Furthermore, the value of the closed convex envelope at $(\mbf{0}, \mbf{0})$ is $\ul{Q}_n'((\mbf{0}, \mbf{0}); \hat{x}_{a(n)}) = \ul{Q}_n(\hat{x}_{a(n)})$.
    \item \label{proof:tight_relu:3} %With the neighborhood $B_{\epsilon}(\mbf{0},\mbf{0})$, 
    We show that if the coefficients $(u_n^+,u_n^-,u_{n0})$ satisfy the conditions of the proposition, then an optimal solution of \eqref{norm:dual} induces a cut of the form \eqref{norm:cut} that is tight at some point $(\ti{w}_n^+,\ti{w}_n^-)\in \relint(B_{\epsilon}(\mbf{0},\mbf{0}))$ with respect to the closed convex envelope $\ol{\co}(\ul{Q}_n'(\cdot,\cdot;\hat{x}_{a(n)}))$. 
    %\ab{Tightness is typically defined with respect to $\ul{Q}_n'$; here we use tightness with respect to $\ol{\co}(\ul{Q}_n')$ instead. We state this explicitly once here so that it need not be repeated in the following steps. Also, $\ol{\co}(\ul{Q}_n')$ comes up again in \eqref{proof:proj_point} in proof of Part \ref{proof:tight_relu:3} when we use Corollary 3.22 of \citealt{fullner2024lipschitz}.} \sk{Need to discuss this. The proposition statement is tightness with respect to $\ul{Q}_n(\hat{x}_{a(n)})$.} 
    %Note, from Part \ref{proof:tight_relu:2}, that tightness at $(\mbf{0}, \mbf{0})$ for $\ol{\co}(\ul{Q}_n')$ implies tightness at $(\mbf{0}, \mbf{0})$ with respect to $\ul{Q}_n'$. In the original space, this leads to tightness with respect to $\ul{Q}_n$ at $\hat{x}_{a(n)}$.
    \item \label{proof:tight_relu:4} If the cut is tight at $(\ti{w}_n^+, \ti{w}_n^-) \in \relint(B_{\epsilon}(\mbf{0}, \mbf{0}))$ for $\ol{\co}(\ul{Q}_n'(\cdot,\cdot;\hat{x}_{a(n)}))$, then it is also tight at $(\mbf{0}, \mbf{0})$ for $\ol{\co}(\ul{Q}_n'(\cdot,\cdot;\hat{x}_{a(n)}))$. By Part \ref{proof:tight_relu:2}, tightness at $(\mbf{0}, \mbf{0})$ for $\ol{\co}(\ul{Q}_n')$ implies tightness at $(\mbf{0}, \mbf{0})$ with respect to $\ul{Q}_n'$. Projecting to the original space, this yields tightness with respect to $\ul{Q}_n$ at $\hat{x}_{a(n)}$.. Furthermore, the cut written in the original space is a Pareto-optimal $h$-cut over the reference set $\mc{H}_n$ for $h$ defined in \eqref{eq:h_function}.
\end{enumerate}

We first prove \textbf{Part} \ref{proof:tight_relu:1}. By Lemma~2.2 of \citealt{fullnernew}, the function $\ul{Q}_n$ is piecewise polyhedral with finitely many pieces. Observe that the function $\ul{Q}_n'(\cdot,\cdot;\hat{x}_{a(n)})$ can be written as the composition $\ul{Q}_n \circ G$, where the affine mapping
$G:\mathbb{R}^{2d_{a(n)}} \to \mathbb{R}^{d_{a(n)}}$ is defined as $ G(w_n^+,w_n^-)=\hat{x}_{a(n)}+w_n^+-w_n^-$. For notational simplicity, we omit the dependence of the function $G$ on the incumbent $\hat{x}_{a(n)}$. To show that $\ul{Q}_n \circ G$ is piecewise polyhedral, it suffices to prove that its epigraph is a finite union of polyhedra. Indeed,
\begin{align*}
\epi(\ul{Q}_n \circ G)
&= \left\{(w_n^+,w_n^-,\theta_n) : \theta_n \ge \ul{Q}_n\big(G(w_n^+,w_n^-)\big)\right\} \\
&= \left\{(w_n^+,w_n^-,\theta_n) : \big(G(w_n^+,w_n^-),\theta_n\big) \in \epi(\ul{Q}_n)\right\}.
\end{align*}
Define the affine map $V:\mathbb{R}^{2d_{a(n)}+1} \to \mathbb{R}^{d_{a(n)}+1}$ as $V(w_n^+,w_n^-,\theta_n) = \big(G(w_n^+,w_n^-),\theta_n\big)$. 
With this notation, $\epi(\ul{Q}_n \circ G) = V^{-1}\big(\epi(\ul{Q}_n)\big)$. Since $\ul{Q}_n$ is piecewise polyhedral, its epigraph $\epi(\ul{Q}_n)$ can be expressed as a finite union of polyhedra. The preimage of a polyhedron under an affine map is again a polyhedron. Consequently, $V^{-1}\big(\epi(\ul{Q}_n)\big)$ is a finite union of polyhedra and hence $\epi(\ul{Q}_n \circ G)$ is piecewise polyhedral. This completes the proof of the first part.

We next prove \textbf{Part} \ref{proof:tight_relu:2}. Since $\ol{\co}(\ul{Q}_n'(\cdot,\cdot;\hat{x}_{a(n)}))$ is a polyhedral convex function, it admits a finite max-representation. The representation is finite because of Part \ref{proof:tight_relu:1}. Specifically, there exists a finite index set $I$ and affine functions
\begin{align}
    \ell_i(w_n^+,w_n^-)
    = \alpha_i^\top w_n^+ + \beta_i^\top w_n^- + \gamma_i,
    \qquad i\in I, \label{eq:affine_pieces}
\end{align}
such that $\ol{\co}\big(\ul{Q}_n'(w_n^+,w_n^-;\hat{x}_{a(n)})\big) = \max_{i\in I} \ell_i(w_n^+,w_n^-)$.
Now, define 
\begin{align*}
    M := \ol{\co}\big(\ul{Q}_n'(\mbf{0},\mbf{0};\hat{x}_{a(n)})\big) = \max_{i\in I} \ell_i(\mbf{0},\mbf{0}) = \max_{i\in I} \gamma_i,
\end{align*}
where the second equality follows from substituting $(\mbf{0}, \mbf{0})$ in \eqref{eq:affine_pieces}. Let $I^\star := \{ i\in I : \ell_i(\mbf{0},\mbf{0}) = M \}$
denote the set of affine pieces attaining the maximum at $(\mbf{0},\mbf{0})$. If $I^\star = I$, then any $\epsilon>0$ suffices for the ball $B_{\epsilon}(\mbf{0}, \mbf{0})$. Otherwise, we construct $\epsilon$ as follows. Fix any $i\notin I^\star$ (such an $i$ exists since $I^\star \subsetneq I$). Then $\ell_i(\mbf{0},\mbf{0})<M$. Choose any $j\in I^\star$ and consider the affine difference
\begin{align*}
d_{ij}(w_n^+,w_n^-)
:= \ell_j(w_n^+,w_n^-) - \ell_i(w_n^+,w_n^-).
\end{align*}
Since $d_{ij}$ is affine and
\begin{align*}
d_{ij}(\mbf{0},\mbf{0})
&= \ell_j(\mbf{0},\mbf{0}) - \ell_i(\mbf{0},\mbf{0}) \\
&= M - \ell_i(\mbf{0},\mbf{0})
> 0,
\end{align*}
by continuity, there exists $\epsilon_{ij}>0$ such that
\begin{align*}
d_{ij}(w_n^+,w_n^-)>0
\quad \text{for all } (w_n^+,w_n^-)\in B_{\epsilon_{ij}}(\mbf{0},\mbf{0}).
\end{align*}
Equivalently,
\begin{align*}
\ell_i(w_n^+,w_n^-)
< \ell_j(w_n^+,w_n^-)
\quad \text{for all } (w_n^+,w_n^-)\in B_{\epsilon_{ij}}(\mbf{0},\mbf{0}).
\end{align*}
Now define $\epsilon_i := \min_{j\in I^\star} \epsilon_{ij} > 0$. Then, for all $(w_n^+,w_n^-)\in B_{\epsilon_i}(\mbf{0},\mbf{0}),$
\begin{align*}
\ell_i(w_n^+,w_n^-)
&< \ell_j(w_n^+,w_n^-), \qquad \forall j\in I^\star, \\
\ell_i(w_n^+,w_n^-)
&< \max_{j\in I^\star} \ell_j(w_n^+,w_n^-) \\
&\le \max_{m\in I} \ell_m(w_n^+,w_n^-) \\
&= \ol{\co}\big(\ul{Q}_n'(w_n^+,w_n^-;\hat{x}_{a(n)})\big).
\end{align*}
Thus, no affine piece $\ell_i$ with $i\notin I^\star$ can locally define the function in $B_{\epsilon_i}(\mbf{0},\mbf{0})$. Since there are finitely many indices $i\notin I^\star$, let $\epsilon := \min_{i\notin I^\star} \epsilon_i > 0$. Then, for every
$(w_n^+,w_n^-)\in B_{\epsilon}(\mbf{0},\mbf{0}),$ any affine piece that locally defines
$\ol{\co}(\ul{Q}_n'(\cdot,\cdot;\hat{x}_{a(n)}))$ must belong to $I^\star$ and therefore agree at $(\mbf{0},\mbf{0})$.
This proves the first statement in Part \ref{proof:tight_relu:2}. The second statement in this part follows immediately from the proof of Proposition \ref{prop:norm_cut_existence}.

Now, we prove \textbf{Part} \ref{proof:tight_relu:3}. We use Lemma 3.21 of \citealt{fullnernew}, which states that the normalized dual problem \eqref{norm:dual} can be formulated as an LP, and the dual of the LP is given by:
\begin{equation} \label{norm:feas}
\begin{aligned}
    \min_{\lambda_n, w_n^+, w_n^-, \eta_n} \quad &\eta_n \\
    &(\lambda_n, w_n^+, w_n^-) \in \conv(\mc{W}_{\hat{x}_{a(n)}}) \\
    &\eta_n \geq 0, \\
    &u_{n0} \eta_n \geq c_n^{\top} \lambda_n - \hat{\theta}_n, \\
    &\eta_n \pmat{u_n^+ \\ u_n^-} = \pmat{w_n^+ \\ w_n^-} - \pmat{\mbf{0} \\ \mbf{0}}.
\end{aligned}
\end{equation}
Here, the notation $\lambda_n$ denotes the set of variables $(x_n, y_n, (\theta_m)_{m \in C(n)})$, and $c_n^{\top} \lambda_n$ denotes the objective function $f_n(x_n, y_n) + \sum_{m \in C(n)} q_{nm} \theta_m$. Further, the set $\mc{W}_{\hat{x}_{a(n)}}$ is defined as:
\begin{align*}
\mc{W}_{\hat{x}_{a(n)}} := \{(\lambda_n, w_n^+, w_n^-):\;
    &(x_n, y_n) \in H_n(\hat{x}_{a(n)} + w_n^+ - w_n^-) \cap (X_n \times Y_n), \\
    & (w_n^+, w_n^-) \in Z_{\hat{x}_{a(n)}}^{lift}, \\
    &(x_n, \theta_m) \in \Psi_m, \quad \forall m \in C(n)\}.
\end{align*}
Now consider the ball $B_{\epsilon}(\mbf{0}, \mbf{0})$ defined in the proof of Part \ref{proof:tight_relu:2}. The relatively complete recourse assumption $\dom(\ul{Q}_n) = Z_{a(n)} = \prod_k [0, B_k]$ ensures that for every $(w_n^+, w_n^-) \in \conv(Z^{lift}_{\hat{x}_{a(n)}})$, the corresponding state $\hat{x}_{a(n)} + w_n^+ - w_n^-$ lies in the domain of $\ul{Q}_n$, and hence the feasible set $H_n(\hat{x}_{a(n)} + w_n^+ - w_n^-)$ is non-empty. Consequently, for $(w_n^+, w_n^-) \in \conv(Z^{lift}_{\hat{x}_{a(n)}})$, there exists $\lambda_n = (x_n, y_n, (\theta_m)_{m \in C(n)})$ such that $(\lambda_n, w_n^+, w_n^-) \in \mc{W}_{\hat{x}_{a(n)}}$. In other words,
\begin{align*}
    B_{\epsilon}(\mbf{0}, \mbf{0}) \subseteq \Proj_{w_n^+, w_n^-} \conv(\mc{W}_{\hat{x}_{a(n)}}).
\end{align*}
Now fix any $(u_n^+, u_n^-) \in \relint(B_{\epsilon}(\mbf{0}, \mbf{0})) \cap \relint(\conv(Z_{\hat{x}_{a(n)}}^{lift}))$ and choose $u_{n0} \geq \ul{Q}_n'(u_n^+, u_n^-; \hat{x}_{a(n)}) - \hat{\theta}_n$. We claim that the LP dual \eqref{norm:feas} has optimal value $\eta^* \leq 1$. To see this, consider the candidate solution $\eta_n = 1$ and $(w_n^+, w_n^-) = (u_n^+, u_n^-)$. By the inclusion above, and definition of $\mc{W}_{\hat{x}_{a(n)}}$ and $\ul{Q}_n'$, there exists $\lambda_n$ such that $(\lambda_n, u_n^+, u_n^-) \in \conv(\mc{W}_{\hat{x}_{a(n)}})$ and $c_n^\top \lambda_n \leq \ul{Q}_n'(u_n^+, u_n^-; \hat{x}_{a(n)})$. The constraint $\eta_n (u_n^+, u_n^-)^\top = (w_n^+, w_n^-)^\top - (\mbf{0}, \mbf{0})^\top$ is satisfied since $\eta_n = 1$. Further, since $u_{n0} \geq \ul{Q}_n'(u_n^+, u_n^-; \hat{x}_{a(n)}) - \hat{\theta}_n$, the constraint $u_{n0} \eta_n \geq c_n^\top \lambda_n - \hat{\theta}_n$ is also satisfied. Hence $\eta_n = 1$ is feasible, implying $\eta^* \leq 1$.

By Corollary 3.22 of \citealt{fullnernew}, the projection of $((\mbf{0}, \mbf{0}), \hat{\theta}_n)$ onto $\epi(\ol{\co}(\ul{Q}_n'(\cdot, \cdot; \hat{x}_{a(n)})))$ along direction $(u_n^+, u_n^-, u_{n0})$ is given by
\begin{align} \label{proof:proj_point}
    (\ti{w}_n^+, \ti{w}_n^-, \ti{\theta}_n) = (\mbf{0}, \mbf{0}, \hat{\theta}_n) + \eta_n^* (u_n^+, u_n^-, u_{n0}),
\end{align}
and the normalized cut \eqref{norm:cut} supports $\epi(\ol{\co}(\ul{Q}_n'(\cdot, \cdot; \hat{x}_{a(n)})))$ at this point. Since $\eta_n^* \leq 1$ and $(u_n^+, u_n^-) \in \relint(B_{\epsilon}(\mbf{0}, \mbf{0}))$, we have $(\ti{w}_n^+, \ti{w}_n^-) = \eta_n^*(u_n^+, u_n^-) \in \relint(B_{\epsilon}(\mbf{0}, \mbf{0}))$. This completes the proof of Part \ref{proof:tight_relu:3}.

Finally, we prove \textbf{Part} \ref{proof:tight_relu:4}. Recall from Part \ref{proof:tight_relu:2}, that 
\begin{align*}
\ell_j (\mbf{0}, \mbf{0}) = \gamma_j = \ul{Q}_n'(\mbf{0},\mbf{0};\hat{x}_{a(n)}) = \ul{Q}_n(\hat{x}_{a(n)}), \text{ for all } j \in I^\star.
\end{align*}
For simplicity, we denote the value $ \ul{Q}_n'(\mbf{0},\mbf{0};\hat{x}_{a(n)})$ as $Q_0$ in rest of this proof. Now, the function $\ol{\co}(\ul{Q}_n'(\cdot,\cdot;\hat{x}_{a(n)}))$ for $(w_n^+,w_n^-) \in B_{\epsilon}(\mbf{0},\mbf{0})$ can be written as:
\begin{align*}
    \ol{\co}(\ul{Q}_n'(w_n^+,w_n^-;\hat{x}_{a(n)})) = Q_0 + \max_{j \in I^\star} \alpha_j^\top w_n^+ + \beta_j^\top w_n^-  =: Q_0 + p(w_n^+, w_n^-),
\end{align*}
where $p(w_n^+, w_n^-) := \max_{j \in I^\star} \alpha_j^\top w_n^+ + \beta_j^\top w_n^-$. Note that the function $p$ is positively homogeneous, that is,
$p(t w_n^+, t w_n^-) = t p(w_n^+, w_n^-) \text{ for } t \ge 0$. Now, consider the normalized cut obtained in Part \ref{proof:tight_relu:3}. Again, for simplicity, we denote this cut as $\theta_n \geq c^+ w_n^+ + c^- w_n^- + c_0$. Since the cut is valid, 
\begin{align*}
    c^+ w_n^+ + c^- w_n^- + c_0 \leq \ol{\co}(\ul{Q}_n'(w_n^+,w_n^-;\hat{x}_{a(n)})), \quad \forall (w_n^+,w_n^-) \in B_{\epsilon}(\mbf{0},\mbf{0}).
\end{align*}
This implies that 
\begin{align*}
    c^+ w_n^+ + c^- w_n^- + c_0 - Q_0 \leq p(w_n^+, w_n^-), \quad \forall (w_n^+,w_n^-) \in B_{\epsilon}(\mbf{0},\mbf{0}).
\end{align*}
Now, for $(\ti{w}_n^+, \ti{w}_n^-)$ identified in Part \ref{proof:tight_relu:3}, we have $c^+ \ti{w}_n^+ + c^- \ti{w}_n^- + c_0 = \ol{\co}(\ul{Q}_n'(\ti{w}_n^+,\ti{w}_n^-;\hat{x}_{a(n)}))$, therefore, the above implication leads to
\begin{align*}
    c^+ \ti{w}_n^+ + c^- \ti{w}_n^- + c_0 - Q_0 = p(\ti{w}_n^+, \ti{w}_n^-).
\end{align*}
Furthermore, since $(\ti{w}_n^+, \ti{w}_n^-) \in \relint(B_{\epsilon}(\mbf{0},\mbf{0}))$ and also $(\mbf{0}, \mbf{0}) \in B_{\epsilon}(\mbf{0},\mbf{0})$, there exists $\delta > 0$ such that $t (\ti{w}_n^+,\ti{w}_n^-) \in B_{\epsilon}(\mbf{0},\mbf{0})$ for all $t \in [1 - \delta, 1 + \delta]$. For such $t$, we have:
\begin{align*}
    t (c^+ \ti{w}_n^+ + c^- \ti{w}_n^-) + c_0 - Q_0 \leq p(t \ti{w}_n^+, t \ti{w}_n^-) = t p(\ti{w}_n^+, \ti{w}_n^-) = t (c^+ \ti{w}_n^+ + c^- \ti{w}_n^- + c_0 - Q_0).
\end{align*}
This gives:
$c_0 - Q_0 \leq t (c_0 - Q_0)$, which implies that $(c_0 - Q_0)(1 - t) \leq 0$ for all  $t \in [1 - \delta, 1 + \delta]$. Taking $t < 1$ gives $c_0 - Q_0 \leq 0$ and $t > 1$ gives $c_0 - Q_0 \geq 0$, and hence $c_0 - Q_0 = 0$ or $c_0 = Q_0$. This proves that the obtained cut of the form $\theta_n \geq c^+ w_n^+ + c^- w_n^- + c_0$ is tight at $(\mbf{0}, \mbf{0})$ for $\ol{\co}(\ul{Q}_n'(\cdot,\cdot;\hat{x}_{a(n)}))$. By Part \ref{proof:tight_relu:2}, tightness at $(\mbf{0}, \mbf{0})$ for $\ol{\co}(\ul{Q}_n')$ implies tightness at $(\mbf{0}, \mbf{0})$ with respect to $\ul{Q}_n'$. Projecting to the original space, this yields tightness with respect to $\ul{Q}_n$ at $\hat{x}_{a(n)}$. Moreover, the coefficients $(u_n^+, u_n^-, u_{n0})$ selected in Part \ref{proof:tight_relu:3} satisfy the requirements of Proposition~\ref{prop:relu_pareto}, so the normalized cut is also a Pareto-optimal $h$-cut.
\end{proof}
The main insight of Proposition~\ref{prop:tight_relu} is that tight cuts can be attained through an appropriate choice of the core point. In particular, the proposition suggests selecting the core point sufficiently close to $(\mbf{0},\mbf{0})$ in order to obtain a cut that is tight at the incumbent. \ignore{A proof of Proposition~\ref{prop:tight_relu} is provided in Appendix \ref{app:proof_tight_v1}.}We next strengthen Proposition~\ref{prop:tight_relu} by showing that for any vectors $(u_n^+,u_n^-)$ satisfying the relative-interior requirement, we can tune the scalar $u_{n0}$ to produce a tight cut.
\begin{proposition} \label{prop:tight_relu_alpha}
    There exists a scalar $\alpha > 1$ such that for normalization coefficients $(u_n^+, u_n^-, u_{n0})$ satisfying:
    \begin{align*}
        &(u_{n}^+, u_n^-) \in \relint(\conv(Z^{lift}_{\hat{x}_{a(n)}})) \\
        &u_{n0} \geq  \alpha \left(\ul{Q}_n'(u_n^+, u_n^-; \hat{x}_{a(n)}) - \hat{\theta}_n \right),
    \end{align*}
    the optimal dual solution $(\hat{\pi}_n^+, \hat{\pi}_n^-, \hat{\pi}_{n0})$ to \eqref{norm:dual} with $\hat{\pi}_{n0} > 0$ defines a tight ReLU cut at incumbent solution $\hat{x}_{a(n)}$. Furthermore, the cut is Pareto-optimal $h$-cut over the reference set $\mc{H}_n$ for $h$ defined in \eqref{eq:h_function}.
\end{proposition}

The proof of Proposition \ref{prop:tight_relu_alpha} is provided in Appendix \ref{app:proof_tight_relu_alpha}. The proposition shows that, for any core point, one can obtain a tight cut by choosing the coefficient $u_{n0}$ sufficiently large. Together, Propositions \ref{prop:tight_relu} and \ref{prop:tight_relu_alpha} establish normalization as a method to construct cuts that are simultaneously tight and Pareto-optimal. Moreover, our arguments extend to the setting with standard copy constraints (as opposed to ReLU-based constraints), in which case the resulting cuts reduce to linear normalized Lagrangian cuts. Recent literature (e.g., \citealt{yangyang2025}) also proposes an alternative approach to producing tight, Pareto-optimal cuts; we contrast this approach with our method in the next section.

\section{Normalization vs. Regularization} \label{sec:norm-vs-reg}

We first revisit recently proposed regularization-based approaches for constructing strong cuts when dual degeneracy can yield cuts with poor approximation quality. We then relate these approaches to normalization and highlight the advantages of the normalization perspective.

The key idea behind regularization is to characterize (or implicitly represent) the set of all optimal solutions to dual \eqref{relu:dual}, and then select a desired solution by optimizing a regularized objective over this set. This objective augments the original dual objective with additional terms involving the dual variables, weighted by appropriately chosen coefficients that ensure tight, Pareto-optimal cuts. We refer to this approach as regularization in analogy with machine learning, where one augments a loss function with penalty terms to bias the optimizer toward solutions with preferred structure.

Formally, let $\Pi_n(\hat{x}_{a(n)})$ denote the set of all optimal solutions of the ReLU dual \eqref{relu:dual}. Given a core point $(\ti{u}_n^+, \ti{u}_n^-) \in \relint(\conv(Z_{\hat{x}_{a(n)}}^{lift}))$, \citealt{yangyang2025} prove that the cut coefficients resulting from solving the following problem: 
\begin{align} \label{reg:dual}
    \max_{\pi_n^+, \pi_n^- \in \Pi_n(\hat{x}_{a(n)})} 
    \mc{L}^R_n(\pi_n^+, \pi_n^-; \hat{x}_{a(n)}) - (\pi_n^+,\pi_n^-) \pmat{\ti{u}_n^+ \\ \ti{u}_n^-},
\end{align}
produces a tight and Pareto-optimal cut of the form \eqref{cut:ReLU}. This approach is inspired by the classical work of \citealt{magnanti1981accelerating} on accelerating Benders decomposition when LP subproblems have dual degeneracy. \citealt{yangyang2025} show that the set $\Pi_n(\hat{x}_{a(n)})$ can be modeled using the constraint
\begin{align} \label{regYY:opt_set}
    \mc{L}_n^{R}(\pi_n^+, \pi_n^-; \hat{x}_{a(n)}) \geq \ul{Q}_n(\hat{x}_{a(n)}) - \epsilon
\end{align}
where $\epsilon > 0$ is a small value. Since the function  $\mc{L}_n^{R}(\pi_n^+, \pi_n^-; \hat{x}_{a(n)})$ is concave and piecewise linear, it is approximated iteratively using gradient cuts via the level-bundle method (\citealt{lemarechal1995new}). The normalized dual problem \eqref{norm:dual} is also solved using the level-bundle method. Moreover, note that the Lagrangian relaxation $\mc{L}_n$ in the extended space is simply a mixed-integer reformulation of the Lagrangian relaxation $\mc{L}_n^{R}$ in the original space. Consequently, the computational complexity of solving the regularized and normalized duals is comparable.

\citealt{deng2024relu} follow a similar approach to \eqref{reg:dual} but instead of approximating the set $\Pi_n(\hat{x}_{a(n)})$ exactly using convex program of the form \eqref{regYY:opt_set}, they approximate it using an LP. Specifically, instead of using the function $\mc{L}_n^{R}(\pi_n^+, \pi_n^-; \hat{x}_{a(n)})$ defined in \eqref{relu:lagrn_relax}, they use the LP relaxation of $\mc{L}_n^{R}(\pi_n^+, \pi_n^-; \hat{x}_{a(n)})$ with additional constraints to prevent the approximate linear program from becoming infeasible. Furthermore, the authors discuss only the choice of regularization coefficients that yield a bounded linear program, but provide no guarantees regarding cut quality. The motivation for this LP-based approach is to recover extreme points of $\Pi_n(\hat{x}_{a(n)})$, which correspond to facet-defining cuts in the extended space. However, because the LP only approximates the true optimal set, it may introduce spurious extreme points or exclude existing ones, thereby producing weaker cuts. We empirically compare our normalization approach with both regularization-based methods in the computational results section.

Next, we show that any cut attained by the regularization method \eqref{reg:dual} can also be attained by solving the normalized dual problem \eqref{norm:dual} using appropriate normalization coefficients.
\begin{proposition} \label{prop:norm_vs_reg}
    For $(\ti{u}_n^+, \ti{u}_n^-) \in \relint(\conv(Z_{\hat{x}_{a(n)}}^{lift}))$, consider a cut of the form \eqref{cut:ReLU} obtained by solving the regularized dual \eqref{reg:dual}. Then, there exist normalization coefficients $(\hat{u}_n^+, \hat{u}_n^-, \hat{u}_{n0})$ such that the coefficients of the regularized cut (up to scaling) are also optimal to the normalized dual \eqref{norm:dual}.
\end{proposition}

\begin{proof}
We prove the result in the following steps:
\begin{enumerate}[label=8.\arabic*]
    \item \label{proof:reg_prop} For a given core point $(\ti{u}_n^+, \ti{u}_n^-)$, consider the regularized dual in \eqref{reg:dual}. Let $(\ti{\pi}_n^+, \ti{\pi}_n^-)$ be any optimal solution of this dual problem. We first show that the resulting cut of the form:
    \begin{align} \label{eq:regularized_cut}
        \theta_n \geq \mc{L}^R_n(\ti{\pi}_n^+, \ti{\pi}_n^-; \hat{x}_{a(n)}) - (\ti{\pi}_n^+, \ti{\pi}_n^-) \cdot  \begin{pmatrix} w_n^+ \\ w_n^-\end{pmatrix}, 
    \end{align}
    supports $\ol{\co}(\ul{Q}_n'(\cdot, \cdot; \hat{x}_{a(n)}))$ on the segment $\mc{S} = \br{t (\ti{u}_n^+, \ti{u}_n^-) : t \in [0, \epsilon']}$, for some $\epsilon' > 0$. 
    
    Since $(\ti{\pi}_n^+,\ti{\pi}_n^-)\in\Pi_n(\hat{x}_{a(n)})$, we have by (strong) duality that
    \begin{align*}
        \mc{L}^R_n(\ti{\pi}_n^+,\ti{\pi}_n^-;\hat{x}_{a(n)})
        &= \ul{Q}_n(\hat{x}_{a(n)}) = \ol{\co}(\ul{Q}_n'(\mbf{0},\mbf{0};\hat{x}_{a(n)})),
    \end{align*}
    and therefore \eqref{eq:regularized_cut} passes through
    $
        \big((\mbf{0},\mbf{0}),\,\ol{\co}(\ul{Q}_n'(\mbf{0},\mbf{0};\hat{x}_{a(n)}))\big)
    $
    in the lifted space.
    
    Moreover, by Part~\ref{proof:tight_relu:2} of Proposition~\ref{prop:tight_relu}, there exist $\epsilon>0$, a finite index set $I^\star$, and vectors $\{(\alpha_j,\beta_j)\}_{j\in I^\star}$ such that, for all $(w_n^+,w_n^-)\in B_\epsilon(\mbf{0},\mbf{0})\cap \conv(Z_{\hat{x}_{a(n)}}^{lift})$,
    \begin{equation} \label{eq:local_max_rep_norm_vs_reg}
        \ol{\co}(\ul{Q}_n'(w_n^+,w_n^-;\hat{x}_{a(n)}))
        = \ul{Q}_n(\hat{x}_{a(n)}) + \max_{j\in I^\star}\left\{\alpha_j^\top w_n^+ + \beta_j^\top w_n^- \right\}.
    \end{equation}
    In particular, every affine function on the right-hand side of \eqref{eq:local_max_rep_norm_vs_reg} is a supporting hyperplane of $\ol{\co}(\ul{Q}_n'(\cdot,\cdot;\hat{x}_{a(n)}))$ that is tight at $(\mbf{0},\mbf{0})$. Since \eqref{reg:dual} maximizes a linear functional over the set $\Pi_n(\hat{x}_{a(n)})$ of dual-optimal cut coefficients (all of which are tight at $(\mbf{0},\mbf{0})$), the optimality of $(\ti{\pi}_n^+,\ti{\pi}_n^-)$ implies that
    \begin{align*}
        (-\ti{\pi}_n^+,-\ti{\pi}_n^-)\cdot \pmat{\ti{u}_n^+\\ \ti{u}_n^-}
        &= \max_{j\in I^\star} \left\{(\alpha_j,\beta_j)\cdot \pmat{\ti{u}_n^+\\ \ti{u}_n^-}\right\} =: \kappa.
    \end{align*}
    Choose
    \begin{align*}
        \epsilon' \;:=\; \min\left\{\,1,\ \frac{\epsilon}{\left\|(\ti{u}_n^+,\ti{u}_n^-)\right\|_2}\right\},
    \end{align*}
    so that $t (\ti{u}_n^+,\ti{u}_n^-)\in B_\epsilon(\mbf{0},\mbf{0})\cap \conv(Z_{\hat{x}_{a(n)}}^{lift})$ for all $t\in[0,\epsilon']$.
    Then for any $t\in[0,\epsilon']$, using \eqref{eq:local_max_rep_norm_vs_reg} and the definition of $\kappa$, we obtain
    \begin{align*}    \ol{\co}\big(\ul{Q}_n'(t\,\ti{u}_n^+,t\,\ti{u}_n^-;\hat{x}_{a(n)})\big)
        &= \ul{Q}_n(\hat{x}_{a(n)}) + \max_{j\in I^\star}\left\{t\,(\alpha_j,\beta_j)\cdot \pmat{\ti{u}_n^+\\ \ti{u}_n^-}\right\} \\
        &= \ul{Q}_n(\hat{x}_{a(n)}) + t\,\kappa \\
        &= \mc{L}^R_n(\ti{\pi}_n^+,\ti{\pi}_n^-;\hat{x}_{a(n)}) - (\ti{\pi}_n^+,\ti{\pi}_n^-)\cdot \pmat{t\,\ti{u}_n^+\\ t\,\ti{u}_n^-}.
    \end{align*}
    Hence, \eqref{eq:regularized_cut} is tight (and therefore supports $\ol{\co}(\ul{Q}_n'(\cdot,\cdot;\hat{x}_{a(n)}))$) at every point of the segment $\mc{S}$.
    \item \label{proof:norm_prop} Next, we consider a point $(\hat{u}_n^+, \hat{u}_n^-) \in \mc{S} \cap B_{\epsilon}(\mbf{0}, \mbf{0})$ where the ball $B_{\epsilon}(\mbf{0}, \mbf{0})$ is defined in Proposition \ref{prop:tight_relu}. Such a point always exists because by choosing $t$ sufficiently small, we can ensure that $t (\ti{u}_n^+, \ti{u}_n^-) \in B_{\epsilon}(\mbf{0}, \mbf{0})$. The scalar $\hat{u}_{n0}$ is then chosen to satisfy the requirements in Proposition \ref{prop:tight_relu}, that is, $\hat{u}_{n0} = \ul{Q}_n'(\hat{u}_n^+, \hat{u}_n^-; \hat{x}_{a(n)}) - \hat{\theta}_n$. Now, by Corollary~3.22 of \citealt{fullnernew} (see also discussion around \eqref{proof:proj_point}), the normalized dual \eqref{norm:dual} with normalization coefficients $(\hat{u}_n^+, \hat{u}_n^-, \hat{u}_{n0})$ produces a cut of the form \eqref{norm:cut} that supports $\ol{\co}(\ul{Q}_n'(\cdot, \cdot; \hat{x}_{a(n)}))$ at a point $\eta (\hat{u}_n^+, \hat{u}_n^-)$ for some $0<\eta\le 1$ (and $(\mbf{0}, \mbf{0}) \in \mc{S}$).

    \item \label{proof:norm_opt} Let the optimal solution of the normalized dual be $(\hat{\pi}_n^+, \hat{\pi}_n^-, \hat{\pi}_{n0})$. We know that the resulting cut is tight at $(\mbf{0}, \mbf{0})$; this follows from Proposition~\ref{prop:tight_relu}. This implies that the following equality holds,
    \begin{align*}
        \hat{\pi}_{n0} \ul{Q}_n(\hat{x}_{a(n)}) = \mc{L}_n(\hat{\pi}_n^+, \hat{\pi}_n^-, \hat{\pi}_{n0}; \hat{x}_{a(n)}).
    \end{align*}
    Therefore, the optimal objective value in the normalized dual is $\hat{\pi}_{n0} (\ul{Q}_n(\hat{x}_{a(n)}) - \hat{\theta}_n)$ for the solution $(\hat{\pi}_n^+, \hat{\pi}_n^-, \hat{\pi}_{n0})$.

    \item \label{proof:norm_support} From Part \ref{proof:norm_prop}, the normalized cut supports $\ol{\co}(\ul{Q}_n'(\cdot, \cdot; \hat{x}_{a(n)}))$ at $\eta (\hat{u}_n^+, \hat{u}_n^-)$, which means that:
    \begin{align*}
        \ul{Q}_n(\hat{x}_{a(n)}) - \frac{\eta}{\hat{\pi}_{n0}}(\hat{\pi}_n^+, \hat{\pi}_n^-)\cdot \pmat{\hat{u}_n^+ \\ \hat{u}_n^-} = \ol{\co}\big(\ul{Q}_n'(\eta \hat{u}_n^+, \eta \hat{u}_n^-; \hat{x}_{a(n)})\big).
    \end{align*}

    \item \label{proof:reg_suport} From Part \ref{proof:reg_prop}, the regularized cut \eqref{eq:regularized_cut}  also supports $\ol{\co}(\ul{Q}_n'(\cdot, \cdot; \hat{x}_{a(n)}))$ at $\eta (\hat{u}_n^+, \hat{u}_n^-)$, which implies that 
    \begin{align*}
    \ul{Q}_n(\hat{x}_{a(n)}) - \eta (\ti{\pi}_n^+, \ti{\pi}_n^-) \cdot  \begin{pmatrix} \hat{u}_n^+ \\ \hat{u}_n^-\end{pmatrix} = \ol{\co}\big(\ul{Q}_n'(\eta \hat{u}_n^+, \eta \hat{u}_n^-; \hat{x}_{a(n)})\big).
    \end{align*}
    \item Now, we prove that some scaling of the regularized cut is also optimal to the normalized dual. In particular, the solution of interest is $(\ti{\pi}_n^+, \ti{\pi}_n^-, 1)$. Note, that the regularized cut \eqref{relu_linear:cut} with coefficients $(\ti{\pi}_n^+, \ti{\pi}_n^-)$ is the same as the normalized cut \eqref{cut:lagrn_relu} with coefficients $(\ti{\pi}_n^+, \ti{\pi}_n^-, 1)$. 
    Now, we consider the following scaled solution  $(\hat{\pi}_{n0}  \ti{\pi}_n^+, \hat{\pi}_{n0}  \ti{\pi}_n^-, \hat{\pi}_{n0})$, and on substituting in function $g_n$ defining the normalization constraint, we have
    \begin{align*}
        g_n(\hat{\pi}_{n0}  \ti{\pi}_n^+, \hat{\pi}_{n0}  \ti{\pi}_n^-, \hat{\pi}_{n0})
        &= \hat{\pi}_{n0}  \left( (\ti{\pi}_n^+, \ti{\pi}_n^-) \cdot \begin{pmatrix} \hat{u}_n^+ \\ \hat{u}_n^-\end{pmatrix} + \hat{u}_{n0} \right) \\
        &= \hat{\pi}_{n0}  \frac{1}{\eta} \left( \ul{Q}_n(\hat{x}_{a(n)})  -  \ol{\co}\big(\ul{Q}_n'(\eta \hat{u}_n^+, \eta \hat{u}_n^-; \hat{x}_{a(n)})\big) \right) + \hat{\pi}_{n0} \hat{u}_{n0}  \\
        &= (\hat{\pi}_n^+, \hat{\pi}_n^-)\cdot \pmat{\hat{u}_n^+ \\ \hat{u}_n^-} + \hat{\pi}_{n0} \hat{u}_{n0} \leq 1.
    \end{align*}
    The second equality follows from Part \ref{proof:reg_suport} and the third equality follows from Part \ref{proof:norm_support}. The final inequality follows from the fact that $(\hat{\pi}_n^+, \hat{\pi}_n^-, \hat{\pi}_{n0})$ is a feasible solution to the normalized dual problem. This means that the solution $(\hat{\pi}_{n0}  \ti{\pi}_n^+, \hat{\pi}_{n0}  \ti{\pi}_n^-, \hat{\pi}_{n0})$ is also feasible to the normalized dual problem. Furthermore, since the scaled cut associated with $(\hat{\pi}_n^+, \hat{\pi}_n^-, \hat{\pi}_{n0})$ is tight, using a similar argument as in Part \ref{proof:norm_opt}, it has the same objective value as the optimal solution $(\hat{\pi}_n^+, \hat{\pi}_n^-, \hat{\pi}_{n0})$. So, it must be optimal as well. This completes the proof. 
\end{enumerate}
\end{proof}
The converse, however, does not hold: normalization can generate cuts that cannot be obtained from the regularized problem. In particular, Pareto-optimal cuts that are not tight at the incumbent can be attained via the normalized dual but are excluded by regularization-based approaches. We illustrate this distinction with a small example below.

\begin{example} \label{ex:norm_vs_reg}
    We consider a two-stage MSIP with a scalar state variable $x_{a(n)} \in Z_{a(n)} := [0,3]$ and second-stage value function
    \begin{align*}
        Q_n(x_{a(n)}) \;:=\; \min\{\, x_n \;|\; 1.5\,x_n \ge x_{a(n)},\; x_n \in \{0,1,2,3\}\,\}.
    \end{align*}
    For the incumbent $\hat{x}_{a(n)} = 1$ with $\hat{\theta}=0.1$, Figure~\ref{fig:norm_vs_reg} illustrates the value function together with two valid cuts, shown both in the original space and in the lifted space. The cut depicted with dashed lines is tight at the incumbent and Pareto-optimal. By contrast, the cut depicted with dotted lines is not tight at $\hat{x}_{a(n)}$ but remains Pareto-optimal. The latter cut can be obtained via normalization but not via regularization-based approaches, because its coefficients are infeasible in the regularization problem. In particular, the regularization constraint $\pi_n^+, \pi_n^- \in \Pi_n(\hat{x}_{a(n)})$ or its $\epsilon$-approximation \eqref{regYY:opt_set}, is violated because the dotted cut is not tight.

    \begin{figure}[ht]
        \centering
        \begin{subfigure}{0.48\textwidth}
            \centering
            \raisebox{5.5em}{\includegraphics[width=\textwidth]{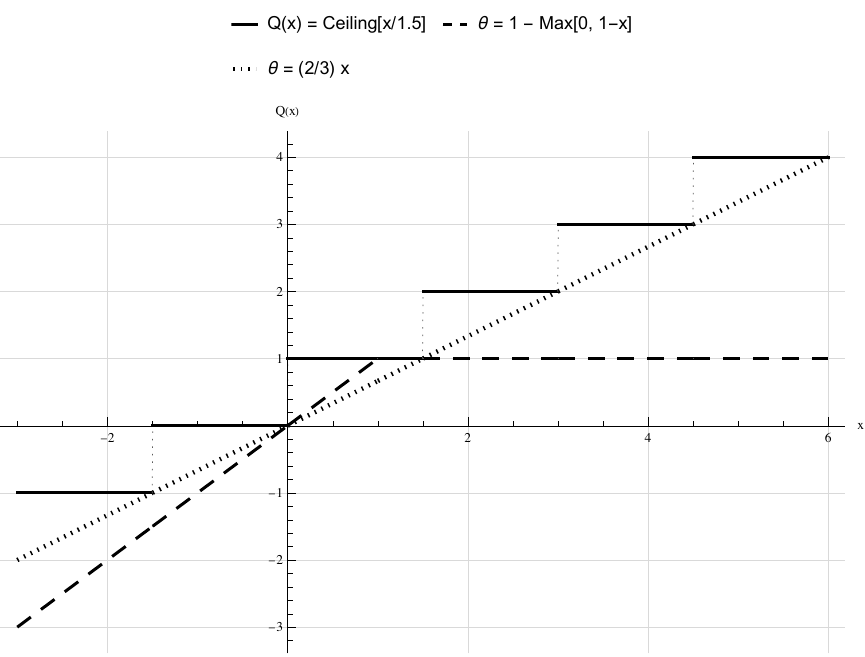}}
        \end{subfigure}
        \hfill
        \begin{subfigure}{0.48\textwidth}
            \centering
            \includegraphics[width=\textwidth]{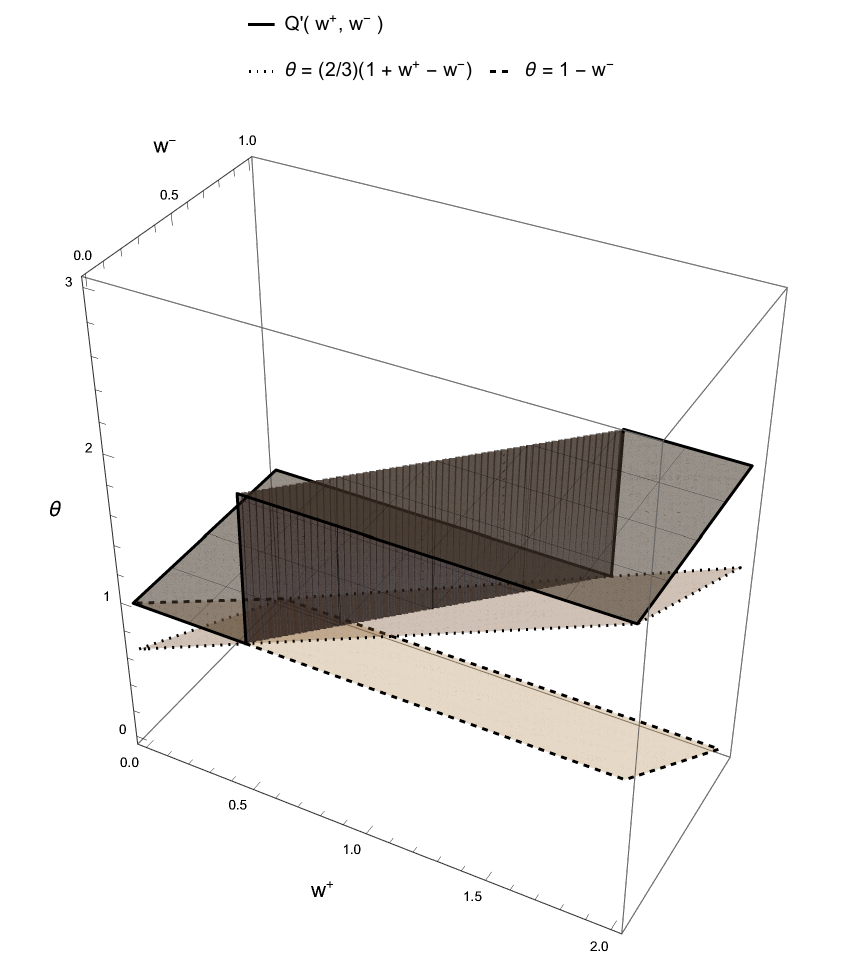}
        \end{subfigure}
        \caption{Left: $Q_n(\cdot)$ and the two cuts over the original domain $Z_{a(n)}$. Right: the corresponding representations in the lifted space $Z^{\mathrm{lift}}_{\hat{x}_{a(n)}}$. The value function is shown in black; the cuts are shown with dotted and dashed boundaries.}
        \label{fig:norm_vs_reg}
    \end{figure}
\end{example}
Proposition \ref{prop:norm_vs_reg} and Example \ref{ex:norm_vs_reg} show that normalization provides an additional flexibility: cuts need not be both tight and Pareto-optimal simultaneously. With an appropriate choice of normalization coefficients, both properties can be enforced simultaneously. However, normalization also permits alternative choices that yield Pareto-optimal cuts that are not tight at the incumbent, while still cutting off the incumbent solution. Another advantage of normalization is that it operates on the epigraph $\epi(Q_m)$, thereby allowing the incorporation of both optimality and feasibility cuts within a unified framework. In contrast, regularization-based approaches work directly with the value function $Q_m$ and therefore rely on the assumption of relatively complete recourse. For ease of exposition, we adopt this assumption throughout the paper; however, our results extend naturally to settings in which relatively complete recourse does not hold.

Our results also extend to the classical Benders/LP setting. For classical Benders cuts, previous work (\citealt{brandenberg2021refined, hosseini2021deepest}) shows that both normalization and the regularization framework of \citealt{magnanti1981accelerating} can attain Pareto-optimal cuts. When multiple such cuts exist, however, whether normalization could achieve the same cut as regularization remained unknown. Proposition \ref{prop:norm_vs_reg} resolves this question: with an appropriate choice of normalization constraints, the coefficients of any Pareto-optimal cut obtained via regularization are also optimal in the normalization dual.

\section{Computational Results} \label{sec:comp}

In this section, we report our computational study to assess how normalization and regularization affect the strength of the generated cuts. Specifically, we compare our normalization approach with the two regularization-based methods of \citealt{deng2024relu} and \citealt{yangyang2025}. Since neither implementation is publicly available, we reimplement both approaches.

For \citealt{deng2024relu}, we follow their LP-based approximation of the set $\Pi_n(\hat{x}_{a(n)})$ in \eqref{reg:dual}. When this LP becomes infeasible, we add the no-good cut proposed in their paper to restore feasibility. To ensure boundedness of the LP approximation, we select the objective coefficients $(\tilde{u}_n^+,\tilde{u}_n^-)$ in \eqref{reg:dual} using their Strategy~2, which is designed to guarantee boundedness.

For \citealt{yangyang2025}, several algorithmic details of the level-bundle implementation are not specified. In particular, the dual model—approximated iteratively via gradient cuts—requires explicit bounds on the dual variables to prevent infeasibility or unboundedness of the approximation. We therefore impose artificial bounds chosen as sufficiently large multiples of the integer L-shaped cut coefficients. The level-bundle method also requires an initial point; we set it to the scaled L-shaped coefficients, which, in most instances, satisfy the regularization constraint \eqref{regYY:opt_set}.

We evaluate the proposed and existing methods on two-stage instances of the dynamic capacity allocation problem (DCAP) and the capacitated lot-sizing problem (CLSP), using the formulations and data-generation procedures of \citealt{deng2024relu} and \citealt{fullnernew}, respectively. Our implementation and datasets are available at \url{https://github.com/akulbansal5/RNorm.git}. To the best of our knowledge, it is the first publicly available open-source implementation that ensures (asymptotic) convergence for general MSIPs and supports both normalization and regularization for generating strong cuts. In contrast, the SDDP package (\citealt{dowson2021sddp}) implements Lagrangian cuts that do not guarantee convergence for general MSIPs and can be weak due to dual degeneracy. The open-source implementation of \citealt{fullnernew} employs normalization to strengthen Lagrangian cuts but does not guarantee convergence for general MSIPs.

The stochastic programs and the cut-generation methods are implemented in Julia 1.9 using the JuMP package (\citealt{dunning2017jump}). All optimization models are solved using Gurobi 12.0. Our cut-generation methods are implemented using a multi-cut strategy; we also tested a single-cut implementation, but it consistently performed worse than the multi-cut approach. The stopping criterion is triggered when one of the following conditions is met: (a) the optimality gap falls below 0.1\%, (b) 5000 iterations are reached, (c) a 3600-second time limit is exceeded, or (d) no improvement larger than $10^{-9}$ is observed in both the lower and upper bounds for 10 consecutive iterations. Criterion (d) is included to prevent the methods from reaching the time limit when the bounds have stalled.

Both the normalized dual \eqref{norm:dual} and regularized dual \eqref{reg:dual} are solved using the level bundle method (\citealt{lemarechal1995new}). To ensure a fair comparison between these approaches, we use identical parameters for the level bundle method: a convergence tolerance of $10^{-2}$ and a maximum of 300 iterations. The parameter $\epsilon$ in set $\Pi_n(\hat{x}_{a(n)})$ approximation \eqref{regYY:opt_set} is also set to $10^{-2}$.

The choice of the core point is made to satisfy the conditions in Proposition \ref{prop:core_point}. For components $k$ of the incumbent vector that lie on the boundary, we set $u_{nk}^+ = 10^{-3}$ and $u_{nk}^- = 0$ when $\hat{x}_{nk} = 0$. Similarly, when $\hat{x}_{nk} = B_k$, we set $u_{nk}^- = 10^{-3}$ and $u_{nk}^+ = 0$. This core point selection is identical for both the regularization and normalization methods. For the normalization method, the scalar $u_{n0}$ is set to $\ul{Q}_n(\hat{x}_{a(n)}) - \hat{\theta}_n + \epsilon$, where $\epsilon = 10^{-6}$.

We first report the results for the DCAP instances in Table \ref{tab:dcap_altFalse}. The columns $I$, $J$, $N$, and $S$ describe the instance characteristics: $I$ denotes the number of resources, $J$ the number of tasks to which resources are assigned, $N$ the number of scenarios, and $S$ the number of time periods. For each instance class, we randomly sample three instances and report the average results across them. All instances considered are two-stage problems; the parameter $S$ affects the instance size but does not change the number of stages. The column ``App.'' indicates the approach used: ``Norm'' for normalization, ``R-LP'' for LP-based regularization proposed by \citealt{deng2024relu}, and ``Reg'' for \citealt{yangyang2025}'s regularization approach, where the dual in \eqref{reg:dual}-\eqref{regYY:opt_set} is solved exactly using the level-bundle method. The columns ``Iter'' and ``Time'' report the number of iterations and computation time (in seconds), respectively. The columns ``UB'' and ``LB'' denote the upper bound and lower bound at termination, and ``Gap (\%)'' reports the relative optimality gap. Column ``D-Iter'' denotes the average number of iterations required to solve the dual problem.

From Table~\ref{tab:dcap_altFalse}, we observe that normalization reaches the target gap (below \(0.1\%\)) in fewer decomposition iterations (``Iter") than regularization, indicating that it produces stronger cuts and a more accurate approximation of the value function. However, the regularization approach (``Reg'') typically achieves lower overall solution times. This time advantage is primarily due to the cost of solving the dual problem: the regularized dual is cheaper to solve than the normalized dual, and its associated level-bundle procedure usually requires fewer dual iterations (D-Iter). As a result, even when regularization requires more iterations (Iter) in the decomposition algorithm, it can still yield shorter overall solution times. This effect is particularly visible for \((I,J,N,S)=(3,4,100,5)\) and \((4,5,100,6)\), where normalization uses fewer outer iterations and fewer dual iterations, yet has a larger total solution time than regularization. We attribute this to the proximal step in the level-bundle method, which solves a quadratic program that is more difficult to solve, partly due to the additional variable \(\pi_{n0}\) present in the normalized dual.

Next, we report results for the CLSP instances in Table~\ref{tab:clsp_altFalse}. The columns $P$ and
$N$ denote the number of products and scenarios, respectively. For each $(P,N)$ class, we consider
three randomly sampled instances and report results averaged over these three instances. The columns ``App.'', ``Iter'', ``D-Iter'', ``Time'', ``UB'', ``LB'', and ``Gap (\%)'' are defined as in the previous tables. 

The results in Table~\ref{tab:clsp_altFalse} indicate that normalization outperforms the regularization-based approach on these instances. The method ``R-LP'' makes little
progress within the time limit, terminating with large gaps (approximately $58\%$--$73\%$) and weak
lower bounds. In contrast, both ``Reg'' and ``Norm'' reduce the optimality gap by orders of
magnitude and typically achieve gaps below $0.1\%$ on most instances. Moreover, ``Norm'' generally
requires fewer decomposition iterations and attains better gaps, suggesting
that the normalized method leads to stronger cuts. For larger instances ($P \in \{10,20\}$), both ``Reg'' and ``Norm'' become substantially more
expensive and do not always reach the $0.1\%$ threshold; nevertheless, ``Norm'' still tends to
produce tighter final gaps than ``Reg'' at better or comparable run times.
Unlike the DCAP results, ``Norm'' is also faster than ``Reg'' on the CLSP instances. This is mainly because, for CLSP, the normalized dual usually converges in fewer iterations (``D-Iter") than the regularized dual, and the resulting reduction in the dual solve times improves overall solution time.

For large-scale instances, solving the dual problem required to generate ReLU cuts—whether via the normalized dual~\eqref{norm:dual} or a regularized dual~\eqref{reg:dual}—can become a major computational bottleneck, a phenomenon widely reported in the stochastic programming literature (e.g., \citealt{zou2019stochastic}, \citealt{chen2022generating}, \citealt{fullnernew}, \citealt{yangyang2025}). To reduce this cost, we adopt an enhancement strategy inspired by the alternating cut criterion of \citealt{angulo2016improving}, originally proposed to enhance the integer L-shaped method (\citealt{laporte1993integer}). Rather than generating ReLU cuts at every iteration—which is expensive—we alternate them with cheaper cuts, namely Benders cuts obtained from the LP relaxation of the subproblems. In particular, we compute a ReLU cut only when the Benders cut from the LP relaxation fails to cut off the incumbent solution. This strategy has been shown to significantly improve performance on MSIPs (see \citealt{bansal2024computational}). We report results under this alternating strategy in Tables~\ref{tab:dcap_altTrue} and~\ref{tab:clsp_altTrue} for the DCAP and CLSP instances, respectively. The ``Prop.'' column indicates the proportion of ReLU cuts among all cuts added, including Benders cuts.

Table~\ref{tab:dcap_altTrue} reports results for DCAP instances with the alternating cut generation approach. Comparing these results to those in Table~\ref{tab:dcap_altFalse}, we observe significant improvements in both solution times and final optimality gaps across most instances. The ``Prop.'' column reveals that ReLU cuts are needed in only a small fraction of iterations: for smaller instances such as $(I,J,N,S)=(2,2,10,4)$ and $(2,3,10,4)$, Benders cuts alone suffice (Prop.\ = 0.00--0.03), meaning the problem can be solved without generating expensive ReLU cuts. For larger instances, the proportion increases but remains modest, reaching at most 0.38 for $(I,J,N,S)=(4,5,10,6)$ with the ``Norm'' approach. This translates to Benders cuts successfully cutting off the incumbent solution approximately 62\%--100\% of the time for the ``Reg'' and ``Norm'' approaches. The ``D-Iter'' values appear lower in Table~\ref{tab:dcap_altTrue} compared to Table~\ref{tab:dcap_altFalse}, but this is because ``D-Iter'' now averages over both ReLU cut iterations (which require dual solves) and Benders cut iterations (which require 0 dual iterations). While the alternating criterion does increase the total number of decomposition iterations (e.g., for $(I,J,N,S)=(4,5,10,6)$, ``Reg'' increases from 50 to 66 iterations and ``Norm'' from 28 to 53 iterations) due to the weaker Benders cuts requiring more iterations to close the gap, the computational savings from cheaper cuts per iteration result in significantly faster overall solution times. For instance, solution times for $(I,J,N,S)=(4,5,10,6)$ improve from 1001 to 331 seconds for ``Reg'' and from 1389 to 910 seconds for ``Norm'', despite the increase in decomposition iterations.

The results in Table~\ref{tab:clsp_altTrue} show a more mixed impact of the alternating criterion on CLSP instances compared to DCAP. The ``Prop.'' column reveals that, unlike DCAP instances, CLSP instances require ReLU cuts in a much larger fraction of iterations (typically 60\%--90\%), meaning Benders cuts fail to cut off the incumbent solution in most iterations. Consequently, the alternating strategy provides limited computational savings from cheaper Benders cuts, and we observe modest improvements or slight increases in solution times for smaller instances. However, for larger instances that hit the time limit, the alternating criterion combined with the normalization method yields significant improvements in final optimality gaps. For example, for $(P,N)=(20,10)$ and $(20,100)$, ``Norm'' achieves gaps of 1.91\% and 3.81\%, respectively, compared to gaps of 2.42\% and 6.35\% when the alternating criterion is not applied (as seen in Table~\ref{tab:clsp_altFalse}). Consistent with the DCAP results, the alternating criterion increases the total number of decomposition iterations compared to the non-alternating approach (e.g., for $(P,N)=(5,10)$, ``Reg'' increases from 35 to 46 iterations and ``Norm'' from 29 to 35 iterations), as weaker Benders cuts require more iterations to close the gap.

\begin{table}[t]
\centering
\small
\setlength{\tabcolsep}{4pt}
\renewcommand{\arraystretch}{1.05}
\begin{tabular}{cccc l r r r r r r}
\hline
$I$ & $J$ & $N$ & $S$ & App. & Iter & Time & UB & LB & Gap(\%) & D-Iter \\
\hline
2 & 2 & 10  & 4 & R-LP & 17  & 12   & 1031.57 & 1031.33 & 0.021 & 0  \\
 &  &   &  & Reg  & 6   & 8    & 1031.41 & 1031.24 & 0.016 & 12 \\
 &  &  &  & Norm    & 5   & 16   & 1031.43 & 1030.67 & 0.075 & 14 \\
\hline
2 & 2 & 100 & 4 & R-LP & 6   & 16   & 1094.71 & 1093.72 & 0.09  & 0  \\
 &  & &  & Reg  & 6   & 29   & 1094.28 & 1093.32 & 0.088 & 13 \\
 &  & &  & Norm    & 5   & 35   & 1094.37 & 1093.71 & 0.063 & 13 \\
\hline
2 & 3 & 10  & 4 & R-LP & 24  & 13   & 1499.58 & 1496.82 & 0.196 & 0  \\
& &  &  & Reg  & 10  & 11   & 1497.28 & 1497.11 & 0.012 & 16 \\
& &  &  & Norm    & 7   & 11   & 1497.57 & 1496.95 & 0.041 & 16 \\
\hline
2 & 3 & 100 & 4 & R-LP & 17  & 36   & 1675.32 & 1675.01 & 0.019 & 0  \\
 &  & &  & Reg  & 8   & 53   & 1675.31 & 1674.26 & 0.063 & 15 \\
 &  & &  & Norm    & 7   & 47   & 1675.31 & 1674.87 & 0.026 & 15 \\
\hline
3 & 4 & 10  & 5 & R-LP & 162 & 1695 & 2171.00 & 2127.54 & 1.89  & 0  \\
 &  &   &  & Reg  & 24  & 104  & 2154.60 & 2154.29 & 0.014 & 30 \\
 &  &   &  & Norm    & 16  & 147  & 2154.96 & 2153.82 & 0.053 & 33 \\
\hline
3 & 4 & 100 & 5 & R-LP & 102 & 2486 & 2207.11 & 2201.32 & 0.252 & 0  \\
 &  &  &  & Reg  & 15  & 527  & 2204.29 & 2203.25 & 0.047 & 28 \\
 &  &  &  & Norm    & 12  & 613  & 2204.35 & 2203.40 & 0.042 & 25 \\
\hline
4 & 5 & 10  & 6 & R-LP & 214 & T & 2663.60 & 2493.28 & 6.341 & 0  \\
 &  &  &  & Reg  & 50  & 1001 & 2621.57 & 2619.88 & 0.065 & 46 \\
 & &   &  & Norm    & 28  & 1389 & 2621.37 & 2619.74 & 0.062 & 74 \\
\hline
4 & 5 & 100 & 6 & R-LP & 87  & T & 3047.66 & 2896.04 & 4.953 & 0  \\
 & & &  & Reg  & 26  & 2990 & 3011.89 & 2993.50 & 0.569 & 45 \\
 &  & &  & Norm    & 21  & 3315 & 3006.23 & 3000.85 & 0.172 & 42 \\
\hline
\end{tabular}
\caption{Comparison of normalization and regularization-based approaches on DCAP instances}
\label{tab:dcap_altFalse}
\end{table}

\begin{table}[t]
\centering
\small
\setlength{\tabcolsep}{4pt}
\renewcommand{\arraystretch}{1.05}
\begin{tabular}{cc l r r r r r r}
\hline
$P$ & $N$ & App. & Iter & Time & UB & LB & Gap(\%) & D-Iter \\
\hline
3  & 2   & R-LP & 408 & T & 741.95 & 232.90 & 68.247 & 0  \\
  &    & Reg  & 17  & 6    & 583.84 & 583.61 & 0.046  & 14 \\
 &    & Norm    & 11  & 6    & 546.19 & 545.90 & 0.044  & 13 \\
\hline
5  & 10  & R-LP & 200 & T & 861.80 & 359.72 & 58.248 & 0  \\
  &   & Reg  & 35  & 79   & 804.32 & 803.64 & 0.085  & 35 \\
 &  & Norm    & 29  & 43   & 804.26 & 803.50 & 0.095  & 21 \\
\hline
5  & 50  & R-LP & 149 & T & 868.72 & 360.60 & 58.483 & 0  \\
 &  & Reg  & 28  & 248  & 814.40 & 814.01 & 0.048  & 37 \\
 &  & Norm    & 21  & 94   & 814.48 & 813.83 & 0.081  & 19 \\
\hline
5  & 100 & R-LP & 130 & T & 866.44 & 360.67 & 58.365 & 0  \\
 & & Reg  & 33  & 667  & 811.02 & 810.29 & 0.090  & 36 \\
 & & Norm    & 20  & 164  & 811.03 & 810.37 & 0.081  & 17 \\
\hline
10 & 10  & R-LP & 117 & T & 2539.38 & 696.84  & 72.521 & 0  \\
 &  & Reg  & 80  & T & 1838.53 & 1770.76 & 3.677  & 94 \\
 &  & Norm    & 65  & T & 1834.51 & 1801.39 & 1.802  & 57 \\
\hline
10 & 50  & R-LP & 103 & T & 2552.69 & 697.84  & 72.660 & 0  \\
 &  & Reg  & 40  & T & 1845.30 & 1726.71 & 6.431  & 80 \\
&   & Norm    & 35  & T & 1839.71 & 1778.45 & 3.342  & 59 \\
\hline
10 & 100 & R-LP & 93  & T & 2539.51 & 697.98  & 72.514 & 0  \\
 &  & Reg  & 30  & T & 1860.03 & 1715.24 & 7.797  & 63 \\
&  & Norm    & 26  & T & 1853.86 & 1770.69 & 4.488  & 57 \\
\hline
20 & 10  & R-LP & 81  & T & 3900.79 & 1331.81 & 65.850 & 0  \\
 &  & Reg  & 56  & T & 3128.01 & 2848.67 & 8.930  & 99 \\
 &  & Norm    & 50  & T & 3057.22 & 2983.12 & 2.419  & 98 \\
\hline
20 & 50  & R-LP & 69  & T & 3877.48 & 1331.86 & 65.648 & 0  \\
 &  & Reg  & 30  & T & 3150.86 & 2809.64 & 10.828 & 92 \\
 &   & Norm    & 24  & T & 3108.54 & 2988.41 & 3.861  & 96 \\
\hline
20 & 100 & R-LP & 62  & T & 3888.87 & 1331.87 & 65.751 & 0  \\
 &  & Reg  & 27  & T & 3171.94 & 2793.52 & 11.916 & 92 \\
 & & Norm    & 20  & T & 3160.86 & 2959.74 & 6.351  & 80 \\
\hline
\end{tabular}
\caption{Comparison of normalization and regularization-based approaches on CLSP instances}
\label{tab:clsp_altFalse}
\end{table}

\begin{table}[t]
\centering
\small
\setlength{\tabcolsep}{4pt}
\begin{tabular}{rrrrllrrrrrr}
\toprule
$I$ & $J$ & $N$ & $S$ & App. & Iter & Time & UB & LB & Gap (\%) & D-Iter & Prop. \\
\midrule
2 & 2 & 10  & 4 & R-LP  & 9   & 5    & 1031.4  & 1031.4  & 0      & 0  & 0.00 \\
& & & & Reg   & 9   & 5    & 1031.40 & 1031.40 & 0.000  & 0  & 0.00 \\
& & & & Norm  & 9   & 5    & 1031.40 & 1031.40 & 0.000  & 0  & 0.00 \\
\midrule
2 & 2 & 100 & 4 & R-LP  & 7   & 7    & 1094.28 & 1094.26 & 0.001  & 0  & 0.00 \\
& & & & Reg   & 7   & 6    & 1094.28 & 1094.26 & 0.001  & 0  & 0.00 \\
& & & & Norm  & 7   & 6    & 1094.28 & 1094.26 & 0.000  & 0  & 0.00 \\
\midrule
2 & 3 & 10  & 4 & R-LP  & 10  & 7    & 1497.27 & 1497.27 & 0      & 0  & 0.08 \\
& & & & Reg   & 10  & 6    & 1497.27 & 1497.27 & 0.000  & 1  & 0.03 \\
& & & & Norm  & 10  & 6    & 1497.27 & 1497.27 & 0.000  & 1  & 0.03 \\
\midrule
2 & 3 & 100 & 4 & R-LP  & 19  & 19   & 1675.31 & 1675.27 & 0.002  & 0  & 0.42 \\
& & & & Reg   & 10  & 14   & 1675.31 & 1675.27 & 0.002  & 1  & 0.06 \\
& & & & Norm  & 10  & 12   & 1675.31 & 1675.27 & 0.000  & 1  & 0.06 \\
\midrule
3 & 4 & 10  & 5 & R-LP  & 123 & 140  & 2155.27 & 2153.57 & 0.08   & 0  & 0.77 \\
& & & & Reg   & 33  & 40   & 2154.60 & 2153.76 & 0.038  & 6  & 0.26 \\
& & & & Norm  & 32  & 67   & 2154.61 & 2153.98 & 0.000  & 12 & 0.30 \\
\midrule
3 & 4 & 100 & 5 & R-LP  & 59  & 255  & 2204.69 & 2202.38 & 0.105  & 0  & 0.59 \\
& & & & Reg   & 20  & 161  & 2204.33 & 2203.76 & 0.026  & 4  & 0.13 \\
& & & & Norm  & 20  & 176  & 2204.30 & 2203.64 & 0.000  & 4  & 0.13 \\
\midrule
4 & 5 & 10  & 6 & R-LP  & 268 & 1696 & 2625.66 & 2612.68 & 0.49   & 0  & 0.76 \\
& & & & Reg   & 66  & 331  & 2621.65 & 2620.10 & 0.059  & 12 & 0.34 \\
& & & & Norm  & 53  & 910  & 2621.06 & 2619.73 & 0.001  & 39 & 0.38 \\
\midrule
4 & 5 & 100 & 6 & R-LP  & 222 & T & 3006.03 & 2999.35 & 0.225  & 0  & 0.81 \\
& & & & Reg   & 46  & 1381 & 3005.08 & 3003.36 & 0.058  & 12 & 0.30 \\
& & & & Norm  & 45  & 3070 & 3006.16 & 3002.92 & 0.001  & 13 & 0.29 \\
\bottomrule
\end{tabular}
\caption{Alternating criterion applied to normalization and regularization-based approaches on DCAP instances}
\label{tab:dcap_altTrue}
\end{table}

\begin{table}[!htbp]
\centering
\small
\setlength{\tabcolsep}{4pt}
\begin{tabular}{rrlrrrrrrr}
    \toprule
    $P$ & $N$ & App. & Iter & Time (s) & UB & LB & Gap (\%) & D-Iter & Prop \\
    \midrule
    3  & 2   & R-LP  & 448 & T & 662.68  & 380.24  & 43.032 & 0  & 0.997 \\
     & &  Reg   & 14  & 7    & 583.87  & 583.68  & 0.034  & 14 & 0.818 \\
     & &  Norm  & 11  & 7    & 583.79  & 583.55  & 0.039  & 9  & 0.688 \\
    \midrule
    5  & 10  & R-LP  & 238 & T & 861.44  & 429.23  & 50.174 & 0  & 0.995 \\
     & &  Reg   & 46  & 77   & 804.19  & 803.78  & 0.052  & 29 & 0.899 \\
     & &  Norm  & 35  & 57   & 804.24  & 803.59  & 0.080  & 18 & 0.901 \\
    \midrule
    5  & 50  & R-LP  & 187 & T & 870.56  & 434.98  & 50.034 & 0  & 0.995 \\
     & &  Reg   & 34  & 254  & 814.40  & 813.82  & 0.072  & 29 & 0.877 \\
     & &  Norm  & 28  & 125  & 814.57  & 813.98  & 0.073  & 16 & 0.899 \\
    \midrule
    5  & 100 & R-LP  & 141 & T & 869.15  & 433.32  & 50.136 & 0  & 0.993 \\
     & &  Reg   & 38  & 593  & 811.05  & 810.29  & 0.093  & 27 & 0.879 \\
     & &  Norm  & 26  & 192  & 811.05  & 810.47  & 0.071  & 14 & 0.872 \\
    \midrule
    10 & 10  & R-LP  & 270 & T & 2046.21 & 971.68  & 52.568 & 0  & 0.986 \\
     & &  Reg   & 105 & T & 1840.43 & 1771.32 & 3.750  & 73 & 0.855 \\
     & &  Norm  & 71  & T & 1832.36 & 1800.79 & 1.723  & 50 & 0.903 \\
    \midrule
    10 & 50  & R-LP  & 268 & T & 2089.39 & 958.68  & 54.101 & 0  & 0.985 \\
     & &  Reg   & 78  & T & 1856.18 & 1733.01 & 6.640  & 39 & 0.765 \\
     & &  Norm  & 46  & T & 1836.98 & 1779.64 & 3.128  & 43 & 0.825 \\
    \midrule
    10 & 100 & R-LP  & 233 & T & 2113.81 & 959.55  & 54.603 & 0  & 0.980 \\
     & &  Reg   & 66  & T & 1863.00 & 1713.85 & 8.015  & 30 & 0.699 \\
     & &  Norm  & 40  & T & 1859.48 & 1772.68 & 4.670  & 36 & 0.777 \\
    \midrule
    20 & 10  & R-LP  & 179 & T & 3854.72 & 1413.01 & 63.342 & 0  & 0.994 \\
     & &  Reg   & 116 & T & 3110.37 & 2857.73 & 8.111  & 47 & 0.581 \\
     & &  Norm  & 68  & T & 3045.22 & 2987.04 & 1.908  & 72 & 0.766 \\
    \midrule
    20 & 50  & R-LP  & 77  & T & 3875.89 & 1411.27 & 63.585 & 0  & 0.987 \\
     & &  Reg   & 83  & T & 3102.65 & 2782.38 & 10.318 & 28 & 0.495 \\
     & &  Norm  & 49  & T & 3101.77 & 2990.30 & 3.593  & 49 & 0.689 \\
    \midrule
    20 & 100 & R-LP  & 78  & T & 3886.72 & 1411.67 & 63.679 & 0  & 0.987 \\
     & &  Reg   & 75  & T & 3105.93 & 2719.35 & 12.448 & 23 & 0.457 \\
     & &  Norm  & 48  & T & 3105.17 & 2986.95 & 3.805  & 42 & 0.631 \\
    \bottomrule
\end{tabular}
\caption{Alternating criterion applied to normalization and regularization-based approaches on CLSP instances.}
\label{tab:clsp_altTrue}
\end{table}

To summarize our computational findings, normalization consistently produces stronger cuts than regularization across both DCAP and CLSP instances, as evidenced by fewer decomposition iterations required to reach the target optimality gap. However, solving the normalized dual can be more expensive than solving the regularized dual, leading to mixed results in overall solution times depending on the problem structure. On DCAP instances, regularization often achieves faster solution times despite requiring more iterations, while on CLSP instances, normalization's advantage in both cut strength and dual convergence typically results in faster overall performance. The alternating cut criterion substantially improves performance on DCAP instances by leveraging cheaper Benders cuts, reducing solution times by up to 50\% while maintaining or improving optimality gaps. On CLSP instances, where Benders cuts are less effective, the alternating criterion provides more modest benefits but still yields notable improvements in final gaps for larger instances that hit the time limit, particularly when combined with normalization.

\section{Conclusion} \label{sec:conclusion}
We studied the problem of weak (and potentially ineffective) Lagrangian cuts that arise from dual degeneracy in decomposition methods for multistage stochastic integer programs with mixed-integer state variables. Building on the ReLU-dual framework of \citealt{deng2024relu} and the normalization framework of \citealt{fullnernew}, we introduced a normalized version of the ReLU dual that selects dual solutions through additional normalization constraints. This yields cut coefficients that, in practice, lead to stronger cuts and improved value-function approximations.

On the theoretical side, we established that normalized ReLU cuts can be interpreted directly in the original state space and that normalization can be used to recover strong cuts despite dual degeneracy. In particular, we introduced a notion of Pareto-optimality for nonlinear ReLU cuts in the original space and showed that normalization produces Pareto-optimal cuts under this definition. We also proved that there exists a choice of normalization coefficients that yields cuts that are tight at the current incumbent, providing guidance for selecting normalization coefficients when tightness is desired. Finally, we clarified the connection to recently proposed regularization strategies: any cut obtainable from regularization of the optimal ReLU-dual set can also be obtained via normalization (up to scaling), while normalization is strictly more flexible because it can generate Pareto-optimal cuts that are not necessarily tight at the incumbent, yet still separate the incumbent solution.

Our computational study on DCAP and CLSP instances demonstrates that the theoretical advantages of normalization translate into practical benefits: normalized cuts consistently require fewer decomposition iterations to reach the target gap, confirming their superior strength. The combination of normalization with the alternating cut criterion further enhances performance, particularly when Benders cuts can effectively contribute to the decomposition process. Together, these results suggest that normalization provides a flexible approach to generating strong Lagrangian cuts that can improve the efficiency of decomposition methods for multistage stochastic integer programs.

Several directions remain for future work. First, developing adaptive, instance-dependent strategies for selecting normalization coefficients—beyond the baseline rule used in our experiments—could further improve robustness and speed. Second, combining normalization with cut-selection policies and more advanced bundle-management techniques may reduce the overhead of solving the normalized dual on difficult instances.

\appendix

\section*{Appendix}

\section{Proof of Proposition~\ref{prop:lifted_lagrangian}} \label{app:lifted_lagrangian}

Recall that the subproblem $\ul{Q}'_n(\mbf{0}, \mbf{0}; \hat{x}_{a(n)})$ is reformulated with copy constraints in \eqref{lifted_sub:copy}, and the corresponding Lagrangian relaxation, dual, and Lagrangian cut in the lifted space are given by \eqref{prob:lifted_relax}, \eqref{dual:lifted_lagrn_00}, and \eqref{cut:lifted_lagrn_00}, respectively. The key observation is that the Lagrangian relaxation \eqref{prob:lifted_relax} is a reformulation of the ReLU-based relaxation \eqref{relu:lagrn_relax} using the linearization in \eqref{cut:relu_linear}. Specifically, by substituting the constraints from \eqref{cut:relu_linear} into the ReLU-based relaxation, we obtain $\mc{L}^R_n(\pi_n^+, \pi_n^-; \hat{x}_{a(n)}) = \mc{L}^O_n(\pi_n^+, \pi_n^-; \hat{x}_{a(n)})$. Similarly, with the constraints $(w_n^+, w_n^-) \in Z_{\hat{x}_{a(n)}}^{lift}$ from \eqref{cut:relu_linear}, the Lagrangian cut \eqref{cut:lifted_lagrn_00} is equivalent to the ReLU cut \eqref{cut:ReLU}. This establishes that the ReLU Lagrangian cut \eqref{cut:ReLU}, generated at $\hat{x}_{a(n)}$ for $\epi_{Z_{a(n)}}\left(\un{Q}_n\right)$, corresponds to the Lagrangian cut \eqref{cut:lifted_lagrn_00} generated at $\left(\mathbf{0}, \mathbf{0}\right)$ for the lifted epigraphical set $\epi_{Z^{lift}_{\hat{x}_{a(n)}}} \left(\ul{Q}'_n(\cdot, \cdot; \hat{x}_{a(n)})\right)$. 

For optimal dual multipliers obtained by solving the ReLU dual \eqref{relu:dual}, we have $\mc{L}^R_n(\pi_n^+, \pi_n^-; \hat{x}_{a(n)}) = \ul{Q}_n(\hat{x}_{a(n)})$. Since $\mc{L}^R_n(\pi_n^+, \pi_n^-; \hat{x}_{a(n)}) = \mc{L}^O_n(\pi_n^+, \pi_n^-; \hat{x}_{a(n)})$ and $\ul{Q}_n(\hat{x}_{a(n)}) = \ul{Q}'_n(\mbf{0}, \mbf{0}; \hat{x}_{a(n)})$, it follows that $\mc{L}^O_n(\pi_n^+, \pi_n^-; \hat{x}_{a(n)}) = \ul{Q}'_n(\mbf{0}, \mbf{0}; \hat{x}_{a(n)})$. Therefore, the Lagrangian cut \eqref{cut:lifted_lagrn_00} is tight at $\left(\mathbf{0}, \mathbf{0}\right)$ with respect to the lifted value function $\ul{Q}'_n(\cdot, \cdot; \hat{x}_{a(n)})$.

\section{Proof of Proposition \ref{prop:tight_relu_alpha}} \label{app:proof_tight_relu_alpha}

The proof is largely similar to the proof of Proposition \ref{prop:tight_relu}. We showed in that proof that there exits a ball $B_{\epsilon} (\mbf{0}, \mbf{0}) \subseteq \Proj_{w_n^+, w_n^-} \conv(\mc{W}_{\hat{x}_{a(n)}})$. Now, choose  any $(u_{n}^+, u_n^+) \in \relint(\conv(Z_{\hat{x}_{a(n)}}^{lift}))$, then we know there, exists scalar $\hat{\eta}_n$ small enough such that $\hat{\eta}_n \cdot (u_{n}^+, u_n^+) \in B_{\epsilon}(\mbf{0}, \mbf{0})$. Now, if we choose $u_{n0}$ large enough then we can ensure that $\eta_n^* = \hat{\eta}_n$, and using \eqref{proof:proj_point}, we will have $(\ti{w}_n, \ti{w}_n) \in B_{\epsilon}(\mbf{0}, \mbf{0})$. In particular, we let
$    u_{n0} = \frac{1}{\hat{\eta}_n} \left(\ol{\co}(\ul{Q}_n'(u_n^+, u_n^-; \hat{x}_{a(n)})) - \hat{\theta}_n \right).$
Then, following similar reasoning as in the proof of Proposition \ref{prop:tight_relu}, we conclude that the resulting cut is tight and Pareto-optimal.

\bibliographystyle{plainnat}
\bibliography{ref.bib}

@article{gade2014decomposition,
  title={Decomposition algorithms with parametric {G}omory cuts for two-stage stochastic integer programs},
  author={Gade, Dinakar and K{\"u}{\c{c}}{\"u}kyavuz, Simge and Sen, Suvrajeet},
  journal={Mathematical Programming},
  volume={144},
  number={1},
  pages={39--64},
  year={2014},
  publisher={Springer}
}

@article{zou2019stochastic,
  title={Stochastic dual dynamic integer programming},
  author={Zou, Jikai and Ahmed, Shabbir and Sun, Xu Andy},
  journal={Mathematical Programming},
  volume={175},
  number={1},
  pages={461--502},
  year={2019},
  publisher={Springer}
}

@article{benders1962partitioning,
  title={Partitioning Procedures for Solving Mixed-Variable Programming Problems, {N}umerische {M}atkematic 4},
  author={Benders, Jacques F},
  journal={SS8},
  year={1962}
}

@article{chen2022generating,
  title={On generating {L}agrangian cuts for two-stage stochastic integer programs},
  author={Chen, Rui and Luedtke, James},
  journal={INFORMS Journal on Computing},
  volume={34},
  number={4},
  pages={2332--2349},
  year={2022},
  publisher={INFORMS}
}

@article{fischetti2010note,
  title={A note on the selection of {B}enders’ cuts},
  author={Fischetti, Matteo and Salvagnin, Domenico and Zanette, Arrigo},
  journal={Mathematical Programming},
  volume={124},
  pages={175--182},
  year={2010},
  publisher={Springer}
}

@article{laporte1993integer,
  title={The integer {L}-shaped method for stochastic integer programs with complete recourse},
  author={Laporte, Gilbert and Louveaux, Fran{\c{c}}ois V},
  journal={Operations Research Letters},
  volume={13},
  number={3},
  pages={133--142},
  year={1993},
  publisher={Elsevier}
}

@article{sen2005c,
  title={The {$C^3$} theorem and a {$D^2$} algorithm for large scale stochastic mixed-integer programming: Set convexification},
  author={Sen, Suvrajeet and Higle, Julia L},
  journal={Mathematical Programming},
  volume={104},
  pages={1--20},
  year={2005},
  publisher={Springer}
}

@article{sen2006decomposition,
  title={Decomposition with branch-and-cut approaches for two-stage stochastic mixed-integer programming},
  author={Sen, Suvrajeet and Sherali, Hanif D},
  journal={Mathematical Programming},
  volume={106},
  pages={203--223},
  year={2006},
  publisher={Springer}
}

@article{zhang2014finitely,
  title={Finitely convergent decomposition algorithms for two-stage stochastic pure integer programs},
  author={Zhang, Minjiao and K\"u\c{c}\"ukyavuz, Simge},
  journal={SIAM Journal on Optimization},
  volume={24},
  number={4},
  pages={1933--1951},
  year={2014},
  publisher={SIAM}
}

@article{qi2017ancestral,
  title={The ancestral {B}enders’ cutting plane algorithm with multi-term disjunctions for mixed-integer recourse decisions in stochastic programming},
  author={Qi, Yunwei and Sen, Suvrajeet},
  journal={Mathematical Programming},
  volume={161},
  pages={193--235},
  year={2017},
  publisher={Springer}
}

@article{ntaimo2013fenchel,
  title={{F}enchel decomposition for stochastic mixed-integer programming},
  author={Ntaimo, Lewis},
  journal={Journal of Global Optimization},
  volume={55},
  pages={141--163},
  year={2013},
  publisher={Springer}
}

@article{van2024converging,
  title={A converging {B}enders’ decomposition algorithm for two-stage mixed-integer recourse models},
  author={van der Laan, Niels and Romeijnders, Ward},
  journal={Operations Research},
  volume={72},
  number={5},
  pages={2190--2214},
  year={2024},
  publisher={INFORMS}
}

@article{brandenberg2021refined,
  title={Refined cut selection for {B}enders decomposition: applied to network capacity expansion problems},
  author={Brandenberg, Ren{\'e} and Stursberg, Paul},
  journal={Mathematical Methods of Operations Research},
  volume={94},
  number={3},
  pages={383--412},
  year={2021},
  publisher={Springer}
}

@article{romeijnders2024benders,
  title={Benders decomposition with scaled cuts for multistage stochastic mixed-integer programs},
  author={Romeijnders, Ward and van der Laan, Niels},
  year={2024},
  journal={Optimization Online}
}

@article{angulo2016improving,
  title={Improving the integer {L}-shaped method},
  author={Angulo, Gustavo and Ahmed, Shabbir and Dey, Santanu S},
  journal={INFORMS Journal on Computing},
  volume={28},
  number={3},
  pages={483--499},
  year={2016},
  publisher={INFORMS}
}

@article{lemarechal1995new,
  title={New variants of bundle methods},
  author={Lemar{\'e}chal, Claude and Nemirovskii, Arkadii and Nesterov, Yurii},
  journal={Mathematical Programming},
  volume={69},
  pages={111--147},
  year={1995},
  publisher={Springer}
}

@article{pereira1991multi,
  title={Multi-stage stochastic optimization applied to energy planning},
  author={Pereira, Mario VF and Pinto, Leontina MVG},
  journal={Mathematical Programming},
  volume={52},
  pages={359--375},
  year={1991},
  publisher={Springer}
}

@article{blair1982value,
  title={The value function of an integer program},
  author={Blair, Charles E and Jeroslow, Robert G},
  journal={Mathematical Programming},
  volume={23},
  number={1},
  pages={237--273},
  year={1982},
  publisher={Springer}
}

@article{birge1985decomposition,
  title={Decomposition and partitioning methods for multistage stochastic linear programs},
  author={Birge, John R},
  journal={Operations Research},
  volume={33},
  number={5},
  pages={989--1007},
  year={1985},
  publisher={INFORMS}
}

@article{philpott2020midas,
  title={Midas: A mixed integer dynamic approximation scheme},
  author={Philpott, Andrew B and Wahid, Faisal and Bonnans, Joseph Fr{\'e}d{\'e}ric},
  journal={Mathematical Programming},
  volume={181},
  number={1},
  pages={19--50},
  year={2020},
  publisher={Springer}
}

@article{dowson2021sddp,
  title={SDDP. jl: a {J}ulia package for stochastic dual dynamic programming},
  author={Dowson, Oscar and Kapelevich, Lea},
  journal={INFORMS Journal on Computing},
  volume={33},
  number={1},
  pages={27--33},
  year={2021},
  publisher={INFORMS}
}

@article{ahmed2022stochastic,
  title={Stochastic {L}ipschitz dynamic programming},
  author={Ahmed, Shabbir and Cabral, Filipe Goulart and Freitas Paulo da Costa, Bernardo},
  journal={Mathematical Programming},
  volume={191},
  number={2},
  pages={755--793},
  year={2022},
  publisher={Springer}
}

@article{hosseini2021deepest,
  title={Deepest cuts for {B}enders decomposition},
  author={Hosseini, Mojtaba and Turner, John},
  journal={Operations Research},
  volume={73},
  number={5},
  pages={2591--2609},
  year={2025},
  publisher={INFORMS}
}

@article{magnanti1981accelerating,
  title={Accelerating {B}enders decomposition: Algorithmic enhancement and model selection criteria},
  author={Magnanti, Thomas L and Wong, Richard T},
  journal={Operations Research},
  volume={29},
  number={3},
  pages={464--484},
  year={1981},
  publisher={INFORMS}
}

@article{meyer1974existence,
  title={On the existence of optimal solutions to integer and mixed-integer programming problems},
  author={Meyer, Robert R},
  journal={Mathematical Programming},
  volume={7},
  number={1},
  pages={223--235},
  year={1974},
  publisher={Springer}
}

@article{fullner2024lipschitz,
  title={On {L}ipschitz regularization and {L}agrangian cuts in multistage stochastic mixed-integer linear programming},
  author={F{\"u}llner, Christian and Sun, X Andy and Rebennack, Steffen},
  journal={Available at Optimization Online},
  year={2024}
}

@article{fullnernew,
  title={A new framework to generate {L}agrangian cuts in multistage stochastic mixed-integer programming},
  author={F{\"u}llner, Christian and Sun, X Andy and Rebennack, Steffen},
  journal={Optimization Online},
  year={2024},
}

@article{deng2024relu,
  title={{On the ReLU Lagrangian cuts for stochastic mixed integer programming}},
  author={Deng, Haoyun and Xie, Weijun},
  journal={arXiv preprint arXiv:2411.01229},
  year={2024}
}

@article{yangyang2025,
  title={Globally Converging Algorithm for Multistage Stochastic Mixed-Integer
Programs via Enhanced {L}agrangian Cuts},
  author={Hanbin Yang and Haoxiang Yang},
  journal={Optimization Online},
  year={2025},
}

@phdthesis{stursberg2019mathematics,
  title={On the Mathematics of Energy System Optimization: Network Models, Decomposition, and Economic Incentives},
  author={Stursberg, Paul Melvin},
  year={2019},
  school={Technische Universit{\"a}t M{\"u}nchen}
}

@article{bansal2024computational,
  title={A computational study of cutting-plane methods for multi-stage stochastic integer programs},
  author={Bansal, Akul and K{\"u}{\c{c}}{\"u}kyavuz, Simge},
  journal={arXiv preprint arXiv:2405.02533},
  year={2024}
}

@article{dunning2017jump,
  title={JuMP: A modeling language for mathematical optimization},
  author={Dunning, Iain and Huchette, Joey and Lubin, Miles},
  journal={SIAM Review},
  volume={59},
  number={2},
  pages={295--320},
  year={2017},
  doi={10.1137/15m1020575}
}

@incollection{kuccukyavuz2017introduction,
  title={An introduction to two-stage stochastic mixed-integer programming},
  author={K{\"u}{\c{c}}{\"u}kyavuz, Simge and Sen, Suvrajeet},
  booktitle={Leading Developments from INFORMS Communities},
  pages={1--27},
  year={2017},
  publisher={INFORMS}
}

@article{romeijndersastochastic,
  title={Stochastic Mixed-Integer Programming: {A} Survey},
  author={Romeijndersa, Ward and Zhangb, Yihang and Sen, Suvrajeet},
  journal={Optimization Online},
  year={2025}
}

@article{fullner2025stochastic,
  title={Stochastic dual dynamic programming and its variants: {A} review},
  author={F{\"u}llner, Christian and Rebennack, Steffen},
  journal={SIAM Review},
  volume={67},
  number={3},
  pages={415--539},
  year={2025},
  publisher={SIAM}
}

@article{rahmaniani2017benders,
  title={The {B}enders decomposition algorithm: {A} literature review},
  author={Rahmaniani, Ragheb and Crainic, Teodor Gabriel and Gendreau, Michel and Rei, Walter},
  journal={European Journal of Operational Research},
  volume={259},
  number={3},
  pages={801--817},
  year={2017},
  publisher={Elsevier}
}

\end{document}